\documentclass[10pt,a4paper]{article}
\usepackage[utf8]{inputenc}
\usepackage{amsmath}
\usepackage{amsfonts}
\usepackage{amssymb}
\usepackage{amsthm}
\usepackage{color}
\usepackage{enumerate}
\usepackage{marvosym}
\usepackage{hyperref}
 \newtheorem{thm}{Theorem}[section]

 \newtheorem*{thm*}{Theorem}
 \newtheorem{cor}[thm]{Corollary}
 \newtheorem{lem}[thm]{Lemma}
 \newtheorem{prop}[thm]{Proposition}
 \theoremstyle{definition}
 \newtheorem{defn}[thm]{Definition}
 \theoremstyle{remark}
 \newtheorem{rem}[thm]{Remark}
 
 \newtheorem*{ex}{Example}
 
 \numberwithin{equation}{section}
 
\newcommand{\vertiii}[1]{{\left\vert\kern-0.25ex\left\vert\kern-0.25ex\left\vert #1
  \right\vert\kern-0.25ex\right\vert\kern-0.25ex\right\vert}}

\newcommand{\tr}{\operatorname{Tr}}

\newcommand{\re}{\operatorname{Re}}

\newcommand{\VERT}[1]{{\left\vert\kern-0.25ex\left\vert\kern-0.25ex\left\vert #1 
    \right\vert\kern-0.25ex\right\vert\kern-0.25ex\right\vert}}

\begin{document}
\title{Toeplitz operators on non-reflexive Fock spaces}
\author{Robert Fulsche}

\maketitle
\begin{abstract}
We generalize several results on Toeplitz operators over reflexive, standard weighted Fock spaces $F_t^p$ to the non-reflexive cases $p = 1, \infty$. Among these results are the characterization of compactness and the Fredholm property of such operators, a well-known representation of the Toeplitz algebra, a characterization of the essential centre of the Toeplitz algebra. Further, we improve several results related to correspondence theory, e.g.\  we improve previous results on the correspondence of algebras and we give a correspondence theoretic version of the well-known Berger-Coburn estimates.

\medskip
\textbf{AMS subject classification:} Primary: 47L80; Secondary: 47B35, 30H20

\medskip
\textbf{Keywords:} Toeplitz algebras, Fock spaces, Correspondence Theory
\end{abstract}

\section{Introduction}
The theory of Toeplitz operators on Fock spaces experienced significant attention and progress in recent years. In the present work, we are concerned with operators acting on the standard $p$-Fock spaces $F_t^p$ on $\mathbb C^n$, i.e.\ spaces of entire functions which are $p$-integrable with respect to the Gaussian measure $\mu_{2t/p}$, cf.\ Section 2 for details. Basic accounts on these Fock spaces and their Toeplitz operators can be found in \cite{Janson_Peetre_Rochberg1987,Zhu2012}. Many properties of Toeplitz operators and the algebras they generate have been investigated. One common feature among most of these works is that they deal only with Toeplitz operators on the reflexive Fock spaces, i.e.\ $F_t^p$ with $1 < p < \infty$. Nevertheless, there has been some steady interest in understanding the non-reflexive Fock spaces, $F_t^p$ with $p = 1, \infty$, as well as their linear operators, cf.\ \cite{Cho_Zhu2012, Hu_Lv2014, Hu_Virtanen2023, Hu_Virtanen2020}. The aim of the present work is to present an approach which is suitable for both the reflexive and the non-reflexive cases. Simply speaking, we will apply tools from harmonic analysis to extend some results known from the reflexive cases to the non-reflexive world. Along the way, we will also obtain some results which are new even in the Hilbert space case.

Let us briefly discuss some of the obstacles one has to overcome when studying Toeplitz operators on non-reflexive Fock spaces. Details and definitions will be given later. 

As is well-known \cite{Bauer_Isralowitz2012}, a bounded operator $A$ on a reflexive Fock space is compact if and only if it satisfies two properties: First of all, it needs to be contained in the Banach algebra generated by all Toeplitz operators with bounded symbols (which we will refer to as the \emph{(full) Toeplitz algebra}). Secondly, its Berezin transform needs to be a function vanishing at infinity. When passing to non-reflexive Fock spaces, one faces the following two problems: Given a compact operator, it may happen that it is no longer contained in the Toeplitz algebra. Even worse, the Berezin transform may no longer be injective, i.e.\ there are non-trivial bounded linear operators with Berezin transform being constantly zero. These problems of course lead to other problems, e.g.~difficulties in the characterization of Fredholm properties for operators from the Toeplitz algebra. In \cite{Fulsche_Hagger}, it was shown for the reflexive cases that an operator from the Toeplitz algebra is Fredholm if and only if all of its limit operators are invertible. As a matter of fact, an operator from the Toeplitz algebra over a reflexive Fock space is Fredholm if and only if it is invertible in the Toeplitz algebra modulo the ideal of compact operators. Since, again, the compact operators are not necessarily contained in the Toeplitz algebra when the Fock space is non-reflexive, this causes obvious problems. 

Let us now shift from describing obstacles towards describing the outcomes of this paper. To introduce some short-hand notation for this, we will denote by $F_t^p$ the standard Fock spaces and by $\mathcal T^{F_t^p}$ the (full) Toeplitz algebra over this Fock space. The precise definitions of these and other terms used will be given in Section 2 below. It will then be our first goal to present a suitable approach to the Correspondence Theory on non-reflexive Fock spaces.  The basic concepts of Correspondence Theory originate from \cite{Werner1984} and were already successfully used by the present author \cite{Fulsche2020} to study Toeplitz operators in case of reflexive Fock spaces. This part of the paper, which will be carried out in Section 3, really forms the technical heart of the approach to non-reflexive Fock spaces, as most problems arising from non-reflexivity will be explained and overcome in this section. In the next part, Section 4, we will present some applications of the Correspondence Theory. Some of them are already known from \cite{Fulsche2020}, we give some prominent examples: In \cite{Bauer_Isralowitz2012}, W.\ Bauer and J.\ Isralowitz proved the following compactness characterization for Toeplitz operators over Fock spaces:
\begin{thm*}[{\cite{Bauer_Isralowitz2012}}]
Let $A \in \mathcal L(F_t^p)$, where $1 < p < \infty$. Then,
\begin{align*}
A \in \mathcal K(F_t^p) \Longleftrightarrow A \in \mathcal T^{F_t^p} \text{ and } \widetilde{A} \in C_0(\mathbb C^n).
\end{align*}
\end{thm*}
In \cite{Xia}, J. Xia obtained (in principle) the following result:
\begin{thm*}[{\cite{Xia}]}]
It holds true that
\begin{align*}
\mathcal T^{F_t^2} = \overline{\{ T_f^t: ~f \in \operatorname{BUC}(\mathbb C^n)\} }.
\end{align*}
\end{thm*}
We discussed in \cite{Fulsche2020} how both of these results can be easily obtained from Correspondence Theory in the case of reflexive Fock spaces. Here, we will obtain (versions of) these results even over non-reflexive Fock spaces, adding some necessary modifications due to the non-reflexivity. Further, we will discuss, among other results in this direction, a generalization of following well-known fact:
\begin{thm*}[\cite{Berger_Coburn1987}]
It holds true that
\begin{align*}
\operatorname{essCen}(\mathcal T^{F_t^2}) = \overline{\{ T_f^t: ~f \in \operatorname{VO}_\partial (\mathbb C^n)} = \{ T_f^t: ~f \in \operatorname{VO}_\partial (\mathbb C^n)\} + \mathcal K(F_t^2).
\end{align*}
\end{thm*}
We will end that section by presenting a correspondence theoretic version of the Berger-Coburn estimates \cite{Berger_Coburn1994}, generalizing results such as the compactness version of the Berger-Coburn estimates \cite{Bauer_Coburn_Isralowitz2010}. In the rather short Section 5, we add a small contribution to the ``algebra question'' of Correspondence Theory, initiated by the author in \cite{Fulsche2020}. Finally, in Section 6, we derive the Fredholm characterization for operators from the Toeplitz algebra. Initially, this result has been proven for the reflexive Fock spaces:
\begin{thm*}[{\cite{Fulsche_Hagger}}]
Let $1 < p < \infty$ and $A \in \mathcal T^{F_t^p}$. Then, $A$ is Fredholm if and only if $A_x$ is invertible for every $x \in \partial \mathbb C^n$.
\end{thm*}
Here, $\partial \mathbb C^n$ denotes an appropriate boundary of $\mathbb C^n$ (the boundary coming from the maximal ideal space of bounded uniformly continuous functions on $\mathbb C^n$, considered as a compactification of $\mathbb C^n$) and $A_x$ is the operator obtained by ``shifting'' $A$ to this boundary point in an appropriate way. While some of the results from \cite{Fulsche_Hagger} can be carried out in the case $p = 1$ (which will also be important for our discussion), some key arguments in the proof of the above theorem break down in this case. We will replace these arguments by an investigation of the \emph{Banach algebra of all limit operators}. To the best of the author's knowledge, this algebra, which we will call $\mathfrak{lim}\mathcal T^{F_t^p}$ (see Section \ref{subsec:alglim} for the precise definition) has not been studied before, not even in the Hilbert space case. After an important result (Theorem \ref{thm:reconstr_limit}) on the structure of the elements of this algebra, we prove the following result, which is new even in the Hilbert space case. We formulate it now only for the reflexive cases $1 < p < \infty$, the non-reflexive cases needing some technical modifications (cf.\ Theorem \ref{cor:quotientalgebra}).
\begin{thm*}
Let $1 < p < \infty$. Then, as Banach algebras, $\mathcal T^{F_t^p}/\mathcal K(F_t^p) \cong \mathfrak{lim} \mathcal T^{F_t^p}$.
\end{thm*}
Based on the structural result on the elements of $\mathfrak{lim}\mathcal T^{F_t^p}$, Theorem \ref{thm:reconstr_limit}, it is then possible to extend the result on the Fredholm characterization for elements of the full Toeplitz algebra to non-reflexive Fock spaces.

Let us end this introduction by mentioning that some of the results presented in Sections 3 and 4 have already been obtained in the author's PhD thesis. Nevertheless, we have the feeling that our results deserve proper publication. In addition, we decided to give a more streamlined presentation here.

\section{Preliminaries}
On $\mathbb C^n$ we consider the family of probability measures $\mu_t$ given by
\begin{align*}
d\mu_t(z) = \frac{1}{(\pi t)^n} e^{-\frac{|z|^2}{t}}~dz,
\end{align*}
where $| \cdot |$ is the Euclidean norm on $\mathbb C^n \cong \mathbb R^{2n}$, $dz$ the standard Lebesgue measure and $t > 0$. Given any $1 \leq p < \infty$, we define
\begin{align*}
L_t^p := L^p(\mathbb C^n, \mu_{2t/p}).
\end{align*}
The Fock space $F_t^p$ is given by
\begin{align*}
F_t^p = L^p(\mathbb C^n, \mu_{2t/p}) \cap \operatorname{Hol}(\mathbb C^n),
\end{align*}
where $\operatorname{Hol}(\mathbb C^n)$ denotes the entire functions on $\mathbb C^n$. This space is always endowed with its natural $L^p$-norm, i.e.\
\begin{align*}
\| f\|_p^p = \left(\frac{p}{2\pi t}\right)^n \int_{\mathbb C^n} |f(z)|^p e^{-\frac{p}{2t}|z|^2}~dz.
\end{align*}
For $p = \infty$, we set
\begin{align*}
L_t^\infty := \{ f: \mathbb C^n \to \mathbb C; ~f \text{ measurable with } \| f\|_{L_t^\infty} := \| fe^{-\frac{|\cdot|^2}{2t}}\|_\infty < \infty\}.
\end{align*}
Here, two different Fock spaces come into play: First of all, there is
\begin{align*}
F_t^\infty := L_t^\infty \cap \operatorname{Hol}(\mathbb C^n).
\end{align*}
Further, we also consider
\begin{align*}
f_t^\infty := \{ f \in F_t^\infty: fe^{-\frac{|\cdot|^2}{2t}} \in C_0(\mathbb C^n)\}.
\end{align*}
Here, $C_0(\mathbb C^n)$ denotes the continuous functions on $\mathbb C^n$ vanishing at infinity. For each value of $p$, $F_t^p$ is well-known to be a closed subspace of $L_t^p$. Further, $f_t^\infty$ is a closed subspace of $F_t^\infty$.

These Fock spaces are well-studied objects, the basic references being \cite{Janson_Peetre_Rochberg1987, Zhu2012}, where also the following facts can be found.

For $p = 2$, we clearly are in the Hilbert space setting. We will usually write the $F_t^2$ inner product by $\langle \cdot, \cdot\rangle_t$ or even $\langle \cdot, \cdot\rangle$. As is well-known, $F_t^2$ is even a reproducing kernel Hilbert space, the reproducing kernels being given by
\begin{align*}
K_z^t(w) = e^{\frac{w \cdot \overline{z}}{t}}.
\end{align*}
Here, $w \cdot \overline{z} = \sum_{j=1}^n w_j \overline{z_j}$ is the standard sesquilinear product on $\mathbb C^n$. Indeed, the functions $K_z^t$ are contained in any of the spaces $F_t^p$ and $f_t^\infty$, and they span a dense subspace of $F_t^p$ for $1 \leq p < \infty$ and $f_t^\infty$. The normalized reproducing kernels are now defined as
\begin{align*}
k_z^t(w) = \frac{K_z^t(w)}{\| K_z^t\|_{F_t^2}} = e^{\frac{w \cdot \overline{z}}{t} - \frac{|z|^2}{t}}.
\end{align*}
It follows from elementary computations that we indeed have $\| k_z^t\|_{F_t^p} = 1$ for any $1 \leq p \leq \infty$.

Under the $F_t^2$ inner product, the Fock spaces satisfy the following duality relations:
\begin{itemize}
\item For $1 \leq p < \infty$: $(F_t^p)' \cong F_t^q$, where $\frac{1}{p} + \frac{1}{q} = 1$;
\item $(f_t^\infty)' \cong F_t^1$;
\item $(F_t^\infty)'$ strictly contains $F_t^1$.
\end{itemize}
Note that these identifications are not isometric, but they hold with an equivalence of norms. Nevertheless, we will always identify dual spaces in this way. If $A$ is a bounded linear operator on any of the Fock spaces, $A^\ast$ will denote the adjoint, acting on the dual space, with respect to the $F_t^2$ duality.

We will consider $\mathcal L(X)$, the bounded linear operators on $X$, for $X$ any Banach space. By $\mathcal K(X)$ we will denote the compact operators, $\mathcal N(X)$ will be our notation for the ideal of nuclear operators.

For $p = 2$, we of course have the well-known orthogonal projection $P_t \in \mathcal L(L_t^2)$ mapping onto $F_t^2$ by
\begin{align*}
P_t f(z) = \langle f, K_z^t\rangle_t = \frac{1}{(\pi t)^n} \int_{\mathbb C^n} f(w) e^{\frac{z \cdot \overline{w}}{t}} e^{-\frac{|w|^2}{t}}~dw.
\end{align*}
As is well-known, $P_t$ (interpreted as an integral operator) gives rise to a bounded projection on $L_t^p$ mapping onto $F_t^p$ for any $1 \leq p \leq \infty$ with $P_t|_{F_t^p} = \operatorname{Id}$. Hence, for any $f \in L^\infty(\mathbb C^n)$ and $1 \leq p \leq \infty$ the Toeplitz operator $T_f^t$ given by
\begin{align*}
T_f^t: F_t^p \to F_t^p, \quad T_f^t(g) = P_t (fg)
\end{align*}
is well defined and bounded, satisfying $\| T_f^t\|_{F_t^p \to F_t^p} \leq \| P_t\| \| f\|_\infty$. Further, every such Toeplitz operator leaves $f_t^\infty$ invariant. Therefore, we can also consider $T_f^t \in \mathcal L(f_t^\infty)$.

For $z \in \mathbb C^n$ we define the Weyl operator $W_z^t$ by
\begin{align*}
W_z^tg(w) = k_z^t(w) g(w-z). 
\end{align*}
Indeed, $W_z^t = T_{g_z^t}^t$ with
\begin{align*}
g_z^t(w) = e^{\frac{|z|^2}{2t} + \frac{2i \operatorname{Im}(w \cdot \overline{z})}{t}},
\end{align*}
in particular $W_z^t \in \mathcal L(F_t^p)$ for any $1 \leq p \leq \infty$ and $W_z^t \in \mathcal L(f_t^\infty)$.
These operators satisfy the following well-known properties, which we fix as a lemma:
\begin{lem}
\begin{enumerate}[(1)]
\item $W_z^t$ is an isometry on $F_t^p$ for any $1 \leq p \leq \infty$;
\item $W_z^t W_w^t = e^{-\frac{\operatorname{Im}(z \cdot \overline w)}{t}}W_{z+w}^t$ for any $z, w \in \mathbb C^n$; 
\item For $1 \leq p < \infty$, $z \mapsto W_z^t$ is continuous in strong operator topology over $F_t^p$. The same holds true over $f_t^\infty$.
\end{enumerate}
\end{lem}
Note that (1) and (2) above follow from elementary computations, while (3) holds for $1 \leq p < \infty$ by Scheff\'{e}'s lemma, and the case of $f_t^\infty$ can be proven by some standard $\varepsilon$-$\delta$ argument.

We will also frequently encounter the Berezin transform: For $A \in \mathcal L(F_t^p)$ or $A \in \mathcal L(f_t^\infty)$, we set
\begin{align*}
\widetilde{A}(z) := \langle A k_z^t, k_z^t\rangle_t.
\end{align*}
As is well-known, the Berezin transform $A \mapsto \widetilde{A}$ is injective on $\mathcal L(F_t^p)$ for $1 \leq p < \infty$ and on $\mathcal L(f_t^\infty)$. Nevertheless, there exist non-zero operators $A \in \mathcal L(F_t^\infty)$ such that $\widetilde{A} \equiv 0$. For the reader not familiar with this, we give an example of such an operator:
\begin{ex}
We will show how to construct non-trivial continuous linear functionals on $F_t^\infty$ which vanish identically on $f_t^\infty$. Once such a functional is obtained, one can easily set up a non-trivial rank one operator on $F_t^\infty$ which vanishes on $f_t^\infty$, hence has zero Berezin transform (as $k_z^t \in f_t^\infty$). Recall that for $g \in F_t^\infty$, the function $\mathbb C^n \ni z \mapsto g(z)e^{-\frac{|z|^2}{2t}}$ is bounded and continuous, hence extends to a continuous function on the Stone-\v{C}ech compactification $\beta \mathbb C^n$ of $\mathbb C^n$. For a fixed function $f \in F_t^\infty \setminus f_t^\infty$, there exists a point $x \in \beta \mathbb C^n \setminus \mathbb C^n$, the Stone-\v{C}ech boundary, such that the point evaluation of $fe^{-\frac{|\cdot|^2}{2t}}$ at $x$ gives a value different from zero (since $f \in f_t^\infty$ if and only if $f e^{-\frac{|\cdot|^2}{2t}} \in C_0(\mathbb C^n)$). Denote by $\nu_x(g)$ the functional of point evaluation of $ge^{-\frac{|\cdot|^2}{2t}}$ at $x$. Then, $g \mapsto \nu_x(g)$ is a non-trivial bounded linear functional on $F_t^\infty$ (non-trivial since $\nu_x(f) \neq 0$) which vanishes on $f_t^\infty$.
\end{ex}

 For a (suitable) function $f: \mathbb C^n \to \mathbb C$, we set
\begin{align*}
\widetilde{f}^{(t)}(z) := \langle fk_z^t, k_z^t\rangle_t.
\end{align*}
Under suitable growth conditions (in particular for $f \in L^\infty(\mathbb C^n)$), we then obtain $\widetilde{T_f^t} = \widetilde{f}^{(t)}$.

There will be many statements we will consider over any of the Fock spaces. For this reason, we will write $\mathbb F_t$ to denote a generic Fock space with respect to the parameter $t>0$, that is: If we do not state otherwise, $\mathbb F_t \in \{ F_t^p: ~1 \leq p \leq \infty\} \cup \{ f_t^\infty\}$ arbitrary. For two quantities $A, B$ we will write $A \lesssim B$ and $A \gtrsim B$ for $A \leq cB$ and $A \geq cB$, respectively, with some inessential constant $c > 0$. By $A \simeq B$ we will mean $A \lesssim B$ and $A \gtrsim B$.

\section{The Correspondence Theorem}
For any $z \in \mathbb C^n$ and $A \in \mathcal L(\mathbb F_t)$, we set
\begin{align*}
\alpha_z(A) = W_z^t A W_{-z}^t \in \mathcal L(\mathbb F_t).
\end{align*}
Further, for $f \in L^\infty(\mathbb C^n)$ and $z \in \mathbb C^n$ we let
\begin{align*}
\alpha_z (f)(w) = f(w-z).
\end{align*}
We say that subspaces $\mathcal D_0 \subset L^\infty(\mathbb C^n)$ and $\mathcal D_1 \subset \mathcal L(\mathbb F_t)$ are $\alpha$-invariant if $\alpha_z(f) \in \mathcal D_0$ or $\alpha_z(A) \in \mathcal D_1$ for any $f \in \mathcal D_0$ or $A \in \mathcal D_1$, respectively, and any $z \in \mathbb C^n$.

In the following, we always write
\begin{align*}
\mathcal C_1(\mathbb F_t) := \{ A \in \mathcal L(\mathbb F_t): ~z \to \alpha_z(A) \text{ is } \| \cdot\|_{\mathbb F_t \to \mathbb F_t}\text{-continuous}\}.
\end{align*}
The ``classical counterpart'' to $\mathcal C_1(\mathbb F_t)$ is the space of bounded uniformly continuous functions on $\mathbb C^n$, which we will denote by $\operatorname{BUC}(\mathbb C^n)$.

For $S \subset L^\infty(\mathbb C^n)$ we write
\begin{align*}
\mathcal T_{lin}^{\mathbb F_t}(S) := \overline{\operatorname{span}} \{ T_f^t \in \mathcal L(\mathbb F_t); ~f \in S\}.
\end{align*}

In this section, we are going to establish the following theorem:
\begin{thm}\label{Correspondencetheorem}
Assume $\mathcal D_0$ is an $\alpha$-invariant and closed subspace of $\operatorname{BUC}(\mathbb C^n)$. Further, let $A \in \mathcal C_1(\mathbb F_t)$. Then, the following holds true:
\begin{align*}
A \in \mathcal T_{lin}^{\mathbb F_t}(\mathcal D_0) \Longleftrightarrow \widetilde{A} \in \mathcal D_0.
\end{align*}
\end{thm}
Note that this theorem essentially goes back (in the Hilbert space case) to \cite{Werner1984} and has been proven by the author in \cite{Fulsche2020} for the reflexive cases $1 < p < \infty$. There, we made critical use of the so-called convolution formalism of Quantum Harmonic Analysis (which also goes back to \cite{Werner1984}). To some extent, these concepts can be carried over to the non-reflexive case: At least on $f_t^\infty$ most things work out and hence can be carried over to $F_t^1$ and $F_t^\infty$ by duality. Nevertheless, this approach somehow lacks the naturalness it has for $p \in (1, \infty)$, as technical burdens keep stacking up. Therefore, we chose to present a different entrance point here, which is one of the two well-known Berger-Coburn estimates. While initially only proven for the Hilbert space case \cite{Berger_Coburn1994}, they have recently been generalized to any value of $p \in [1, \infty]$ \cite{Bauer_Fulsche2020}.
\begin{thm}[{\cite{Berger_Coburn1994, Bauer_Fulsche2020}}]\label{thm:berger_coburn1}
Let $f: \mathbb C^n \to \mathbb C$ be measurable such that $f K_z^t \in L_t^2$ for any $z \in \mathbb C^n$. Then, for every $s \in (0, t/2)$ there exists a constant $C$, depending only on $n, s$ and $t$ such that:
\begin{align*}
\| T_f^t\|_{\mathbb F_t \to \mathbb F_t} &\leq C \| \widetilde{f}^{(s)}\|_\infty.
\end{align*}
\end{thm}
\begin{cor}\label{cor:toep_c1}
For any $f \in L^\infty(\mathbb C^n)$ it holds true that
\begin{align*}
T_f^t \in \mathcal C_1(\mathbb F_t).
\end{align*}
\end{cor}
\begin{proof}
Clearly, $\widetilde{f}^{(s)} \in \operatorname{BUC}(\mathbb C^n)$ for every $s > 0$ whenever $f \in L^\infty(\mathbb C^n)$ ($\widetilde{f}^{(s)}$ is simply the convolution of $f$ with an appropriate Gaussian). Hence,
\begin{align*}
\| \alpha_z(T_f^t) - T_f^t\| = \| T_{\alpha_z(f) - f}^t\| \leq C \| \alpha_z(\widetilde{f}^{(s)}) - \widetilde{f}^{(s)}\|_\infty.
\end{align*}
By uniform continuity, the right-hand side of this estimate goes to zero when $|z| \to 0$, proving the statement.
\end{proof}
In particular, this proves that $\mathcal T^{\mathbb F_t} \subset \mathcal C_1(\mathbb F_t)$. Here, we denote by $\mathcal T^{\mathbb F_t}$ the Banach subalgebra of $\mathcal L(\mathbb F_t)$ generated by all Toeplitz operators with bounded symbols.

We can now define an $L^1(\mathbb C^n)$ module structure on $\mathcal C_1(\mathbb F_t)$ as follows. For any $A \in \mathcal C_1(\mathbb F_t)$ and $f \in L^1(\mathbb C^n)$, we set
\begin{align*}
f \ast A := \int_{\mathbb C^n} f(z) \alpha_z(A)~dz.
\end{align*}
This integral always exists as a Bochner integral in $\mathcal C_1(\mathbb F_t)$, hence we have $f \ast A \in \mathcal C_1(\mathbb F_t)$. It follows from basic properties of the Bochner integral that we have
\begin{align*}
\alpha_z(f \ast A) = f \ast (\alpha_z(A)) = \alpha_z(f) \ast A.
\end{align*}
It is also imminent that this module structure satisfies
\begin{align*}
\| f \ast A\|_{\mathbb F_t \to \mathbb F_t} \leq \| f\|_{L^1(\mathbb C^n)} \| A\|_{\mathbb F_t \to \mathbb F_t}.
\end{align*}

In what follows, we will set
\begin{align*}
g_s(z) = \frac{1}{(\pi s)^n}e^{-\frac{|z|^2}{s}}.
\end{align*}
Then, elementary estimates show: For any $A \in \mathcal C_1(\mathbb F_t)$ it holds true that
\begin{align*}
\| A - g_s \ast A\|_{\mathbb F_t \to \mathbb F_t} \longrightarrow 0, \quad s \to 0.
\end{align*}
We will also use the following important identity, which holds for any $A \in \mathcal C_{1}(\mathbb F_t)$:
\begin{align*}
g_t \ast A = T_{\widetilde{A}}^t.
\end{align*}
This identity is well-known for $p = 2$ (see e.g. \cite{Berger_Coburn1987}) and can, for arbitrary $\mathbb F_t$, be verified in the same way, i.e.\ by comparing the Berezin transforms of the two operators: 
\begin{align*}
\widetilde{g_t \ast A}(w) &= \int_{\mathbb C^n} g_t(z) \widetilde{\alpha_z(A)}(w)~dz \\
&= \int_{\mathbb C^n} g_t(w) \widetilde{A}(w-z)~dz\\
&= g_t \ast \widetilde{A}(z)\\
&= (\widetilde{A})^{\sim(t)}(z)\\
&= \widetilde{T_{\widetilde{A}}^t}(z).
\end{align*} 
Note that the Berezin transform over $\mathcal L(F_t^\infty)$ is no longer injective, therefore it is not immediately clear that the above reasoning implies $g_t \ast A = T_{\widetilde{A}}^t$ for this case. Indeed, the Berezin transform is injective on $\mathcal C_1(F_t^\infty)$. We will defer showing this for a moment. 

Recall that, by Wiener's approximation theorem, for any $N \in \mathbb N$ there are finitely many constants $c_j^N \in \mathbb C$, $z_j^N \in \mathbb C^n$ such that
\begin{align*}
\| g_{\frac{t}{N}} - \sum_{j} c_j^N \alpha_{z_j^N}(g_t)\|_{L^1(\mathbb C^n)} \leq \frac{1}{N}.
\end{align*}

\begin{proof}[Proof of the Correspondence Theorem \ref{Correspondencetheorem} for $\mathbb F_t \not = F_t^\infty$]
Let $N \in \mathbb N$ and choose the constants $c_j^N, z_j^N$ as above. Then, we have:
\begin{align*}
\| A - T_{\sum_j c_j^N \alpha_{z_j^N}(\widetilde{A})}^t\| &= \| A - \sum_j c_j^N \alpha_{z_j^N} T_{\widetilde{A}}^t\|\\
&= \| A - \sum_{j} c_j^N \alpha_{z_j^N} g_t \ast A\|\\
&\leq \| A - g_{\frac{t}{N}} \ast A\| + \|  g_{\frac{t}{N}} \ast A - \sum_j c_j^N \alpha_{z_j^N} g_t \ast A\|\\
&\leq \| A - g_{\frac{t}{N}} \ast A\| + \| g_{\frac{t}{N}} - \sum_{j} c_j^N \alpha_{z_j^N} g_t\|_{L^1} \| A\|\\
&\leq \| A - g_\frac{t}{N} \ast A\| + \frac{1}{N} \| A\|.
\end{align*}
We therefore obtain that $A$ can be approximated by Toeplitz operators with symbols in $\mathcal D_0$, whenever $\mathcal D_0$ is an $\alpha$-invariant subspace of $\operatorname{BUC}(\mathbb C^n)$ containing $\widetilde{A}$. In particular, $\widetilde{A} \in \mathcal D_0$ implies $A \in \mathcal T_{lin}^{\mathbb F_t}(\mathcal D_0)$.

On the other hand, if $A \in \mathcal T_{lin}^{\mathbb F_t}(\mathcal D_0)$, then $\widetilde{A} \in \mathcal D_0$ follows easily by translation invariance.
\end{proof}
Let us come back to the problem we mentioned earlier, namely the non-injectivity of the Berezin transform on $\mathcal L(F_t^\infty)$. This problem can be resolved by observing the following facts:
\begin{lem}\label{prop:duality}
\begin{enumerate}[(1)]
\item Every $A\in \mathcal C_1(F_t^\infty)$ has a pre-adjoint in $\mathcal C_1(F_t^1)$, i.e.\ an operator $B \in \mathcal L(F_t^1)$ such that $B^\ast = A$.
\item Every $A \in \mathcal C_1(F_t^\infty)$ leaves $f_t^\infty$ invariant.
\item The restriction of the Berezin transform to $\mathcal C_1(F_t^\infty)$ is injective.
\item $f_t^\infty = \{ f \in F_t^\infty; ~z \mapsto W_z^t f \text{ is continuous with respect to the } F_t^\infty \text{-topology}\}$.
\end{enumerate}
\end{lem}

\begin{proof}
Let us consider
\begin{align*}
\mathcal A := \{ A \in \mathcal C_1(F_t^\infty); ~A \text{ has a pre-adjoint in } \mathcal C_1(F_t^1)\}.
\end{align*}
Then, for any $A \in \mathcal A$ we have
\begin{align*}
\widetilde{A} = \overline{\widetilde{B}},
\end{align*}
where $B$ is a pre-adjoint in $\mathcal L(F_t^1)$, and since the Berezin transform is injective on $\mathcal L(F_t^1)$, we get that $\widetilde{A} = 0$ if and only if $B = 0$ if and only if $A = 0$, i.e.\ the Berezin transform is injective on $\mathcal A$. As before, we obtain now $g_t \ast A = T_{\widetilde{A}}^t$ for $A \in \mathcal A$. Now, one can imitate the proof of the Correspondence Theorem, i.e.\ show that any $A \in \mathcal A$ can be approximated by linear combinations of translates of $T_{\widetilde{A}}^t$. Since every Toeplitz operator (with bounded symbol) on $F_t^\infty$ leaves $f_t^\infty$ invariant, we obtain that every operator from $\mathcal A$ leaves $f_t^\infty$ invariant. From this, we can now prove (4): First, we recall that we already know that $z \mapsto W_z^t f$ is continuous in $F_t^\infty$-topology for any $f \in f_t^\infty$, hence we only need to show the reverse inclusion.

Assume that $f \in F_t^\infty$ is such that $z \mapsto W_z^t f$ is continuous in $F_t^\infty$. Then, the operator $f \otimes 1$ clearly has $1 \otimes f$ as its pre-adjoint, i.e.\ $(1 \otimes f)^\ast = f \otimes 1$. Now, $f \otimes 1$ is clearly contained in $\mathcal A$: We have
\begin{align*}
\| \alpha_z &(f \otimes 1) - (f \otimes 1)\| \\
&\leq \| k_z^t - 1\|_{F_t^1} \| f\|_{F_t^\infty} + \| k_z^t\|_{F_t^1} \| f - W_z^t\|_{F_t^\infty} \rightarrow 0, \quad z \to 0. 
\end{align*}
Similarly, the pre-adjoint $1 \otimes f$ is contained in $\mathcal C_1(F_t^1)$. By the previous discussion, $f \otimes 1$ must leave $f_t^\infty$ invariant, i.e
\begin{align*}
(f \otimes 1)(1) = \langle 1, 1\rangle f = f \in f_t^\infty.
\end{align*}
Finally, we now show that any $A \in \mathcal C_1(F_t^\infty)$ leaves $f_t^\infty$ invariant. From this, we then obtain
\begin{align*}
A = (A|_{f_t^\infty})^{\ast \ast},
\end{align*}
in particular $A$ has a pre-adjoint in $F_t^1$, i.e.\ $A \in \mathcal A$. Then, the statement automatically follows from the reasoning above.

So let $A \in \mathcal C_1(F_t^\infty)$ and pick any $f \in f_t^\infty$. As explained above, the rank 1 operator $(1 \otimes f)$ is contained in $\mathcal C_1(F_t^\infty)$. Therefore, 
\begin{align*}
A(f \otimes 1) = (Af \otimes 1) \in \mathcal C_1(F_t^\infty). 
\end{align*}
Recall that for the operator norm of rank one operators $b \otimes a$ with $a \in F_t^1$, $b \in F_t^\infty$ we have
\begin{align*}
\| b \otimes a\|_{F_t^\infty \to F_t^\infty} &= \sup_{f \in F_t^\infty, ~\| f\| \leq 1} \|  b\otimes a(f)\|_{F_t^\infty}\\
&= \sup_{\| f\| \leq 1} |\langle f, a\rangle| \| b\|_{F_t^\infty}\\
&= \| a\|_{(F_t^\infty)'} \| b\|_{F_t^\infty}\\
&\simeq \| a\|_{F_t^1} \| b\|_{F_t^\infty}.
\end{align*}

Therefore, we obtain
\begin{align*}
0 &= \lim_{z \to 0} \| \alpha_z(A(f \otimes 1)) - A(f \otimes 1)\| \\
&= \lim_{z \to 0}\| (W_z^t (Af) \otimes k_z^t) - (Af \otimes 1)\|\\
&\geq \limsup_{z \to 0} | \| W_z^t (Af) \otimes (k_z^t - 1)\| - \| (W_z^t (Af) - Af) \otimes 1 \| |.
\end{align*}
Now, we clearly have
\begin{align*}
\| W_z^t(Af) \otimes (k_z^t - 1)\| &\lesssim \| k_z^t - 1\|_{F_t^1} \| W_z^t (Af)\|_{F_t^\infty} \\
&= \| k_z^t - 1\|_{F_t^1} \| Af\|_{F_t^\infty} \\
&\to 0, \quad z \to 0
\end{align*}
and therefore necessarily
\begin{align*}
\liminf_{z \to 0} \| (W_z^t(Af) - Af) \otimes 1 \| \gtrsim \limsup_{z \to 0} \| 1\|_{F_t^1} \|W_z^t (Af) - Af\|_{F_t^\infty} = 0.
\end{align*}
This shows that $z \mapsto W_z^t Af$ is continuous in $F_t^\infty$ topology, hence $Af \in f_t^\infty$ by (4). We thus obtain $Af_t^\infty \subset f_t^\infty$.
\end{proof}
\begin{rem}
Indeed, once one sees the characterization of $f_t^\infty$ in (4), it is not at all surprising that this holds true. Nevertheless, it seems that this was not observed before. Obtaining a direct proof of (4) should be a nice exercise.
\end{rem}
\begin{proof}[Proof of the Correspondence Theorem \ref{Correspondencetheorem} for $\mathbb F_t = F_t^\infty$] Follows now as in the other cases, having the injectivity of the Berezin transform on $\mathcal C_1(F_t^\infty)$ at hand.
\end{proof}

\begin{cor}\label{cor_corresp}
Let $\mathcal D_1 \subset \mathcal C_1(\mathbb F_t)$ be an $\alpha$-invariant, closed subspace. Then, 
\begin{align*}
\mathcal D_1 = \mathcal T_{lin}^{\mathbb F_t}(\mathcal D_0)
\end{align*}
for some unique $\alpha$-invariant and closed subspace $\mathcal D_0$ of $\operatorname{BUC}(\mathbb C^n)$.
\end{cor}
\begin{proof}
Set
\begin{align*}
\mathcal D_0 := \overline{\{ \widetilde{A}: ~A \in \mathcal D_1\}}.
\end{align*}
Since $\mathcal D_1$ is $\alpha$-invariant, the same is true for $\mathcal D_0$. The proof of the Correspondence Theorem shows
\begin{align*}
\mathcal D_1 \subset \mathcal T_{lin}^{\mathbb F_t}(\mathcal D_0).
\end{align*}
On the other hand, for $f \in \mathcal D_0$ we have $f = \widetilde{A}$ for some $A \in \mathcal D_1$, hence
\begin{align*}
T_f^t = T_{\widetilde{A}}^t = g_t \ast A,
\end{align*}
where the right-hand side was defined as a Bochner integral with values in $\mathcal D_1$, hence is again contained in $\mathcal D_1$. This shows equality of $\mathcal D_1$ and $\mathcal T_{lin}^{\mathbb F_t}(\mathcal D_0)$. The Correspondence Theorem now easily gives uniqueness of $\mathcal D_0$.
\end{proof}

\section{Applications of the Correspondence Theorem}

Here, we give some applications of the Correspondence Theorem. In what follows, we will denote by $\mathcal T^{\mathbb F_t}$ the full Toeplitz algebra, i.e.\ the Banach algebra generated by all Toeplitz operators with $L^\infty(\mathbb C^n)$ symbols over the respective Fock space $\mathbb F_t$. 
\begin{thm}
The following holds true:
\begin{align*}
\mathcal C_1(\mathbb F_t) = \mathcal T^{\mathbb F_t} = \mathcal T_{lin}^{\mathbb F_t}(\operatorname{BUC}(\mathbb C^n)).
\end{align*}
\end{thm}
\begin{proof}
Follows from the Correspondence Theorem with $\mathcal D_0 = \operatorname{BUC}(\mathbb C^n)$.
\end{proof}
\begin{rem}
The equality $\mathcal T^{F_t^2} = \mathcal T_{lin}^{F_t^2}(L^\infty(\mathbb C^n))$ was first obtained by J. Xia in \cite{Xia}. In \cite{Fulsche2020}, we obtained this result in the reflexive cases. 
\end{rem}

\begin{cor}\label{cor:inv_closed}
$\mathcal T^{\mathbb F_t}$ is inverse closed in $\mathcal L(\mathbb F_t)$.
\end{cor}
\begin{proof}
Let $A \in \mathcal T^{\mathbb F_t}$ and $B \in \mathcal L(\mathbb F_t)$ with $B = A^{-1}$. Then, $(\alpha_z(A))^{-1} = \alpha_z(B)$. By continuity of the inversion, we get that $z \mapsto \alpha_z(B)$ is continuous, hence $B \in \mathcal C_1(\mathbb F_t) = \mathcal T^{\mathbb F_t}$.
\end{proof}

The following result extends the characterization of compact operators on Fock spaces, which was first obtained in \cite{Bauer_Isralowitz2012} for the reflexive cases, to the non-reflexive cases. This initial result is contained in part (1) of the theorem.
\begin{thm}\label{compactnesschar}
\begin{enumerate}[(1)]
\item Let $\mathbb F_t \in \{ F_t^p: ~1 < p < \infty\} \cup \{ f_t^\infty\}$. Then, for $A \in \mathcal L(\mathbb F_t)$ we have:
\begin{align*}
A \in \mathcal K(\mathbb F_t) \Longleftrightarrow A \in \mathcal T^{\mathbb F_t} \text{ and } \widetilde{A} \in C_0(\mathbb C^n).
\end{align*}
\item Let $A \in \mathcal L(F_t^1)$ have a pre-adjoint in $\mathcal L(f_t^\infty)$. Then, we have:
\begin{align*}
A \in \mathcal K(F_t^1) \Longleftrightarrow A \in \mathcal T^{F_t^1} \text{ and } \widetilde{A} \in C_0(\mathbb C^n).
\end{align*}
\item Let $A \in \mathcal L(F_t^\infty)$ with $Af_t^\infty \subset f_t^\infty$. Then, we have:
\begin{align*}
A \in \mathcal K(F_t^\infty) \Longleftrightarrow A \in \mathcal T^{F_t^\infty} \text{ and } \widetilde{A} \in C_0(\mathbb C^n).
\end{align*}
\end{enumerate}
\end{thm}
Before we discuss the proof, let us briefly mention that some additional assumptions in the cases $p = 1, \infty$ are indeed necessary: On $F_t^1$, there are clearly compact operators not contained in $\mathcal T^{F_t^1}$. As an example, consider $1 \otimes f$ with $f \in F_t^\infty \setminus f_t^\infty$. One readily checks that this operator is not contained in $\mathcal C_1(F_t^1)$, hence not in $\mathcal T^{F_t^1}$. Further, there are nontrivial operators in $\mathcal L(F_t^\infty)$ the Berezin of which vanishes identically.
\begin{proof}[Proof of Theorem \ref{compactnesschar}]
Recall that the statement for the reflexive cases have been proven in \cite{Fulsche2020} using the Correspondence Theorem. In principle, the proofs here work analogously. Let us, for completeness, quickly go through the details to verify that nothing strange is happening in the non-reflexive cases. First, we discuss why any $f \in C_0(\mathbb C^n)$ gives rise to a compact Toeplitz operator $T_f^t$ even in the non-reflexive cases, which seems to be not that well-known in the literature. Let us denote 
\begin{align*}
\mathcal C\mathcal N(\mathbb F_t) := \{ A \in \mathcal N(\mathbb F_t): ~z \mapsto \alpha_z(A) \text{ is cont. w.r.t } \| \cdot\|_{\mathcal N(\mathbb F_t)}\}.
\end{align*}
It is easily verified that $(1 \otimes 1) \in \mathcal C\mathcal N(\mathbb F_t)$. Further, $\mathcal C\mathcal N(\mathbb F_t)$ is a closed, $\alpha$-invariant subspace of $\mathcal N(\mathbb F_t)$. Hence, the Bochner integral
\begin{align*}
\int_{\mathbb C^n} f(z) \alpha_z(1 \otimes 1)~dz
\end{align*}
exists as an operator in $\mathcal C \mathcal N(\mathbb F_t)$ for every $f \in L^1(\mathbb C^n)$. Further, $\mathcal C \mathcal N(\mathbb F_t) \subset \mathcal C_1(\mathbb F_t)$. Comparing the Berezin transforms, we see that $\int_{\mathbb C^n} f(z) \alpha_z(1 \otimes 1)~dz = (\pi t)^n T_f^t$. Hence, we get that $T_f^t$ is nuclear whenever $f \in L^1(\mathbb C^n)$. In particular, $T_f^t$ is compact for $f \in C_c(\mathbb C^n)$. By approximation, we obtain compactness for any $f \in C_0(\mathbb C^n)$.

Showing that a compact operator $A \in \mathcal C_1(\mathbb F_t)$ has Berezin transform vanishing at infinity follows from standard methods: On $F_t^p$ ($1 \leq p < \infty$) and $f_t^\infty$, the normalized reproducing kernels $k_z^t$ converge weakly to 0 as $|z| \to \infty$, which implies the result for this case. Over $F_t^\infty$, apply the same reasoning to $A|_{f_t^\infty}$. Those steps together show
\begin{align*}
\mathcal T_{lin}^{\mathbb F_t}(C_0(\mathbb C^n)) = \mathcal C_1(\mathbb F_t) \cap \mathcal K(\mathbb F_t).
\end{align*}
Now, apply the Correspondence Theorem.
\end{proof}

We want to add a nice observation which we borrow from \cite{Luef_Skrettingland2021}, where it was presented for the Hilbert space case. Recall that a function $f \in L^\infty(\mathbb C^n)$ is said to be \emph{slowly oscillating} if for every $\varepsilon > 0$ there exists $K > 0$ and $\delta > 0$ such that
\begin{align*}
|f(x) - f(x-y)| < \varepsilon \text{ whenever } |x| > K, ~|y| <\delta.
\end{align*}
Obviously, bounded uniformly continuous functions are slowly oscillating. Even more so, a slowly oscillating function is continuous if and only if it is uniformly continuous. 

Recall Pitt's improvement of the classical Wiener Tauberian theorem:
\begin{thm}[{\cite[Theorem 4.74]{Folland1995}}]
Let $f \in L^\infty(\mathbb R^n)$ be slowly oscillating and $\phi \in L^1(\mathbb R^n)$ a regular function (i.e.\ its Fourier transform vanishes nowhere). Then, if $f \ast \phi \in C_0(\mathbb R^n)$, it follows that $f(x) \to 0$ as $|x| \to \infty$.
\end{thm}
Since the Berezin transform $\widetilde{f}^{(t)}$ is simply the convolution by a Gaussian, which is a regular function (in the sense that its Fourier transform vanishes nowhere), we obtain the following characterization of compactness for Toeplitz operators:
\begin{cor}
Let $f \in L^\infty(\mathbb C^n)$ be slowly oscillating. Then, the following are equivalent over any $\mathbb F_t$:
\begin{enumerate}[(1)]
\item $f(z) \to 0$ as $z \to \infty$;
\item $\widetilde{f}^{(t)} \in C_0(\mathbb C^n)$;
\item $T_f^t$ is compact.
\end{enumerate}
\end{cor}

In the following, we want to consider the quotient of the Toeplitz algebra modulo the compact operators. Since the compact operators are not entirely contained in $\mathcal T^{\mathbb F_t}$ for $\mathbb F_t = F_t^1, F_t^\infty$, we will abbreviate
\begin{align*}
\mathcal T \mathcal K := \mathcal K(\mathbb F_t) \cap \mathcal T^{\mathbb F_t} = \mathcal T_{lin}^{\mathbb F_t}(C_0(\mathbb C^n)).
\end{align*}
Suppressing $\mathbb F_t$ in this notation will not cause any confusion.

The following result is well-known in the Hilbert space case \cite{Berger_Coburn1987} and has recently extended to the reflexive cases \cite{Hagger2021}.We present it now also for the non-reflexive cases.
\begin{thm}\label{esscen}
The essential center of $\mathcal T^{\mathbb F_t}$, i.e.\
\begin{align*}
\operatorname{essCen}(\mathcal T^{\mathbb F_t}) = \{ A \in\mathcal T^{\mathbb F_t}; ~[A, B] \in \mathcal T \mathcal K \text{ for every } B \in \mathcal T^{\mathbb F_t}\},
\end{align*}
is given by
\begin{align*}
\operatorname{essCen}(\mathcal T^{\mathbb F_t}) = \mathcal T_{lin}^{\mathbb F_t}(\operatorname{VO}_\partial(\mathbb C^n)).
\end{align*}
\end{thm}
Here, $\operatorname{VO}_\partial(\mathbb C^n)$ is the space of functions of vanishing oscillation at infinity:
\begin{align*}
\operatorname{VO}_\partial(\mathbb C^n) = \{ f \in C_b(\mathbb C^n); ~\sup_{|w| \leq 1} |f(z) - f(z-w)| \to 0, \quad |z| \to \infty\}.
\end{align*}
It is not difficult to see that $\operatorname{VO}_\partial (\mathbb C^n)$ is an $\alpha$-invariant closed subspace of $\operatorname{BUC}(\mathbb C^n)$.
\begin{proof}[Proof of Theorem \ref{esscen}]
It is simple to show that $\operatorname{essCen}(\mathcal T^{\mathbb F_t})$ is closed and $\alpha$-invariant. Hence, let $\mathcal D_0$ denote the unique $\alpha$-invariant and closed subspace of $\operatorname{BUC}(\mathbb C^n)$ in the sense of Corollary \ref{cor_corresp}. Since the result is already proven in the Hilbert space case, we know by the compactness characterization that for every $f \in \operatorname{VO}_\partial (\mathbb C^n)$ and $g \in \operatorname{BUC}(\mathbb C^n)$ we have $\widetilde{[T_f^t, T_g^t]} \in C_0(\mathbb C^n)$. Hence, using the compactness characterization we have just proven for any choice of $\mathbb F_t$, this gives that $[T_f^t, T_g^t]$ is compact for each $\mathbb F_t$. Thus,
\begin{align*}
 \mathcal T_{lin}^{\mathbb F_t}(\operatorname{VO}_\partial(\mathbb C^n)) \subset \operatorname{essCen}(\mathcal T^{\mathbb F_t}),
\end{align*}
in particular $\operatorname{VO}_\partial(\mathbb C^n) \subset \mathcal D_0$. 

For the other inclusion, we can literally argue as in the Hilbert space case \cite{Berger_Coburn1987}, which we briefly do for completeness: Let $f \in \mathcal D_0$, i.e.\ $T_f^t \in \operatorname{essCen}(\mathcal T^{\mathbb F_t})$. Then, since $W_z^t$ is a Toeplitz operator with bounded symbol,
\begin{align*}
[W_z^t, T_f^t] = W_z^t T_f^t - T_f^t W_z^t = K \in \mathcal T \mathcal K ,
\end{align*}
hence
\begin{align*}
\alpha_z(T_f^t) - T_f^t = KW_{-z}^t \in \mathcal T \mathcal K.
\end{align*}
This in turn implies
\begin{align*}
T_{\widetilde{f}^{(t)}}^t - T_f^t = \int_{\mathbb C^n} g_t(z) (\alpha_z(T_f^t) - T_f^t)~dz \in \mathcal T \mathcal K,
\end{align*}
hence $\left(\widetilde{f}^{(t)} - f\right)^{\sim(t)} \in C_0(\mathbb C^n)$. But such functions $f$ are well-known to be contained in $\operatorname{VO}_\partial(\mathbb C^n)$ \cite{Berger_Coburn1987}.
\end{proof}
We also have the following consequence, which is again well-known in the Hilbert space case.
\begin{cor}
It holds true that
\begin{align*}
\mathcal T_{lin}^{\mathbb F_t}(\operatorname{VO}_\partial(\mathbb C^n)) = \{ T_f^t + K; ~f \in \operatorname{VO}_\partial(\mathbb C^n), ~K \in \mathcal T \mathcal K\}.
\end{align*}
\end{cor}
\begin{proof}
As in the previous proof, one sees that for any $A \in \mathcal T_{lin}^{\mathbb F_t}(\operatorname{VO}_\partial(\mathbb C^n))$ that
\begin{align*}
\alpha_z(A) - A \in \mathcal T\mathcal K,
\end{align*}
hence
\begin{align*}
T_{\widetilde{A}}^t - A = \int_{\mathbb C^n} g_t(z) (\alpha_z(A) - A)~dz \in \mathcal T \mathcal K.
\end{align*}
Therefore,
\begin{align*}
A = T_{\widetilde{A}}^t + K,
\end{align*}
which finishes the proof.
\end{proof}

Recall that the Calkin algebra is the quotient algebra $\mathcal L(X)/\mathcal K(X)$.  It is a Banach algebra when endowed with its canonical quotient norm. This quotient norm is indeed nothing else but the essential norm of the representatives:
\begin{align*}
\| A\|_{ess} = \inf_{K \in \mathcal K(X)} \| A + K\| = \| A + \mathcal K(X)\|_{\mathcal L(X)/\mathcal K(X)}.
\end{align*}
Let us, for simplicity, abbreviate $\mathcal K = \mathcal K(X)$ in the following. Recall again that for $\mathbb F_t \neq F_t^1, F_t^\infty$, $\mathcal K \subset \mathcal T^{\mathbb F_t}$. Hence, we may consider the quotient $\mathcal T^{\mathbb F_t}/\mathcal K = \mathcal C_1(\mathbb F_t)/\mathcal K$ as a subalgebra of the Calkin algebra. Further, since the shifts leave $\mathcal K$ invariant, they descend to a group action in the Calkin algebra:
\begin{align*}
\alpha_z(A + \mathcal K) = \alpha_z(A) + \mathcal K.
\end{align*}
In this case, we obtain the following:
\begin{thm}\label{char:quotient_algebra}
For $\mathbb F_t \neq F_t^1, F_t^\infty$, the following holds true:
\begin{align*}
\mathcal T^{\mathbb F_t}/\mathcal K = \{ A + \mathcal K \in \mathcal L(\mathbb F_t)/\mathcal K; ~z \mapsto \alpha_z(A + \mathcal K) \text{ is cont. w.r.t. } \| \cdot\|_{\mathcal L(\mathbb F_t)/\mathcal K(\mathbb F_t)}\}.
\end{align*}
\end{thm}
\begin{proof}
For $A + \mathcal K \in \mathcal T^{\mathbb F_t}/\mathcal K$, it is clear that the shifts act continuously on them with respect to the norm (keep in mind that $\| A + \mathcal K\| \leq \| A\|_{op}$). Clearly, the set of all $A + \mathcal K$ on which the shifts act norm continuously forms a closed subalgebra of $\mathcal L/\mathcal K$. Denote this closed subalgebra of $\mathcal L(\mathbb F_t)/\mathcal K$ by $\mathcal C_{1, ess}^{\mathbb F_t}$. We can continue as in the proof of the Correspondence Theorem: Define an $L^1$ module action on $\mathcal C_{1,ess}^{\mathbb F_t}$ by
\begin{align*}
f \ast (A + \mathcal K) := \int_{\mathbb C^n} f(z) \alpha_z(A + \mathcal K)~dz,
\end{align*}
considered as a Bochner integral in $\mathcal C_{1,ess}^{\mathbb F_t}$. It is then again not hard to verify that
\begin{align*}
g_t \ast (A + \mathcal K) = T_{\widetilde{A}}^t + \mathcal K,
\end{align*}
as $\widetilde{A + \mathcal K} \in \widetilde{A} + C_0(\mathbb C^n)$. Now, continue as in the proof of the Correspondence Theorem to show equality.
\end{proof}
\begin{rem}
Indeed, from applying the proof of the Correspondence Theorem as it was sketched above, one does not only get the equality $\mathcal T^{\mathbb F_t}/\mathcal K = \mathcal C_{1,ess}^{\mathbb F_t}$, but also a statement analogous to the Correspondence Theorem in the sense of a correspondence between closed, translation-invariant subspaces of $\mathcal T^{\mathbb F_t}/\mathcal K$ and $\operatorname{BUC}(\mathbb C^n)/C_0(\mathbb C^n)$.
\end{rem}
The statement of the following corollary is contained in \cite{Hagger2021} for $1 < p < \infty$. Hence, the statement is only new over $f_t^\infty$.
\begin{cor}
Let $\mathbb F_t \neq F_t^1, F_t^\infty$ and $A \in \mathcal L(\mathbb F_t)$ such that $[A, B] \in \mathcal K$ for every $B \in \mathcal T^{\mathbb F_t}$. Then, $A \in \mathcal T_{lin}^{\mathbb F_t}(\operatorname{VO}_\partial(\mathbb C^n))$.
\end{cor}
\begin{proof}
The assumption clearly implies $A - \alpha_z(A) \in \mathcal K$. Hence, $\alpha_z(A + \mathcal K) = A + \mathcal K$ in $\mathcal L(\mathbb F_t)/\mathcal K$. The previous theorem yields $A + \mathcal K \in \mathcal C_{1,ess}^{\mathbb F_t} = \mathcal T^{\mathbb F_t}/\mathcal K$. Therefore, there are $B \in \mathcal T^{\mathbb F_t}$ and $K \in \mathcal K$ with $A = B + K$, proving $A \in \mathcal T^{\mathbb F_t}$.
\end{proof}
Here is another corollary, which will be more important to us in the following:
\begin{cor}
Let $\mathbb F_t \neq F_t^1, F_t^\infty$ and $A \in \mathcal T^{\mathbb F_t}$ be Fredholm. Then, for any Fredholm regularizer $B \in \mathcal L(\mathbb F_t)$, i.e.\
\begin{align*}
AB - \operatorname{Id} \in \mathcal K,~ BA - \operatorname{Id} \in \mathcal K,
\end{align*}
we have $B \in \mathcal T^{\mathbb F_t}$. 
\end{cor}
\begin{proof}
Being a Fredholm regularizer is just the same as saying that $B + \mathcal K$ is an inverse to $A + \mathcal K$ in the Calkin algebra. Now, since the Calkin algebra is a Banach algebra with unit, inversion is continuous here. Hence, we obtain that $z \mapsto \alpha_z(B + \mathcal K)$ is continuous in the Calkin algebra whenever $z \mapsto \alpha_z(A + \mathcal K)$ is continuous in the Calkin algebra. Theorem \ref{char:quotient_algebra} implies that $B + \mathcal K \in \mathcal T^{\mathbb F_t}/\mathcal K$ whenever $A \in \mathcal T^{\mathbb F_t}$ and $B +\mathcal K$ is an inverse to $\mathcal A + \mathcal K$. As in the proof of the previous corollary, this implies $B \in \mathcal T^{\mathbb F_t}$.
\end{proof}

While the previous corollary does not carry over to the cases $\mathbb F_t = F_t^1, F_t^\infty$, its essence still prevails: The Fredholm theory of operators from $\mathcal T^{\mathbb F_t}$ is entirely contained in $\mathcal T^{\mathbb F_t}$:

\begin{thm}
For any $A \in \mathcal T^{\mathbb F_t}$ we have:
\begin{align*}
A \text{ is Fredholm } \Longleftrightarrow A \text{ is invertible in } \mathcal T^{\mathbb F_t} \text{ modulo } \mathcal T \mathcal K.
\end{align*}
\end{thm}
\begin{proof}
By the previous results, we only need to deal with the cases $\mathbb F_t = F_t^1, F_t^\infty$. Clearly, we only need to prove ``$\Longrightarrow$''. Assume $A \in \mathcal T^{F_t^\infty}$ is Fredholm. By Lemma \ref{prop:duality}, $A$ leaves $f_t^\infty$ invariant and we obtain $A = (A|_{f_t^\infty})^{\ast \ast}$. Since a bounded linear operator on a Banach space is Fredholm if and only if its adjoint is Fredholm, we can conclude that $A|_{f_t^\infty}$ is Fredholm. Therefore, there exist $B \in \mathcal T^{f_t^\infty}$ and $K_1, K_2 \in \mathcal K(f_t^\infty)$ with
\begin{align*}
A|_{f_t^\infty} B = \operatorname{Id} + K_1, \quad BA|_{f_t^\infty} = \operatorname{Id} + K_2.
\end{align*}
Clearly, we have $B^{\ast \ast}$, $K_1^{\ast \ast}$, $K_2^{\ast \ast} \in \mathcal T^{F_t^\infty}$, $K_1^{\ast \ast}$ and $K_2^{\ast \ast}$ are still compact and
\begin{align*}
A B^{\ast \ast} = \operatorname{Id} + K_1^{\ast \ast}, \quad B^{\ast \ast} A = \operatorname{Id} + K_2^{\ast \ast}.
\end{align*}
This settles the case of $\mathbb F_t = F_t^\infty$. For $p = 1$ and $A \in \mathcal T^{F_t^1}$, we can apply the same construction to $A^\ast \in \mathcal T^{F_t^\infty}$ to obtain
\begin{align*}
A B^\ast = \operatorname{Id} + K_1^\ast, \quad B^\ast A = \operatorname{Id} + K_2^\ast,
\end{align*}
where $B$ is a Fredholm regularizer of $(A^\ast)|_{f_t^\infty}$.
\end{proof}

Let us end this section by discussing correspondence versions of the Berger-Coburn estimates. Recall that the Toeplitz operator $T_f^t$ with (possibly unbounded) symbol $f$ is defined as
\begin{align*}
D(T_f^t) := \{ g \in F_t^p: ~fg \in L_t^p\}, \quad T_f^t(g) := P_t(fg).
\end{align*}
As we already noted above (cf.\ Theorem \ref{thm:berger_coburn1}), one of the two Berger-Coburn estimates has been proven to hold true over any Fock space by W. Bauer and the present author in \cite{Bauer_Fulsche2020}. The other Berger-Coburn estimate, also obtained initially in \cite{Berger_Coburn1994} for the case $p= 2$, can also be carried over to the general case:
\begin{thm}
Let $f: \mathbb C^n \to \mathbb C$ such that $fK_z^t \in L_t^2$ for any $z \in \mathbb C^n$. Then, there exist constants $C > 0$ (depending on $s, t, p, n$) such that for any $s \in (\frac{t}{2}, 2t)$:
\begin{align*}
\| \widetilde{f}^{(s)}\|_\infty &\leq C \| T_f^t\|_{\mathbb F_t \to \mathbb F_t}.
\end{align*}
\end{thm}
Let us very briefly sketch the original proof for the Hilbert space setting and then discuss why this proof still works for arbitrary $\mathbb F_t$.

First of all, there is nothing to prove for $s \geq t$. For $s < t$, Berger and Coburn obtained the inequalities by considering the operators
\begin{align*}
P_k := \sum_{|\alpha| = k} e_\alpha^t \otimes e_\alpha^t,\\
T_0^{(s)} := \sum_{k=0}^\infty \left( 1 - \frac{t}{s}\right)^k P_k.
\end{align*}
Here, $e_\alpha^t$ are the elements of the standard orthonormal basis of $F_t^2$, i.e.\
\begin{align*}
e_\alpha^t(z) = \sqrt{\frac{1}{t^{|\alpha|}\alpha!}} z^\alpha, \quad z \in \mathbb C^n, \alpha\in \mathbb N_0^n,
\end{align*}
using standard multi-index notation. Clearly, $P_k$ is of finite rank. Further, for $s \in (\frac{t}{2}, t)$, the sum defining $T_0^{(s)}$ converges in trace norm, hence $T_0^{(s)}$ is a trace-class operator. Now, upon computing the trace of the product $T_0^{(s)} T_f^t$, they noticed that, up to a constant not depending on $f$,
\begin{align*}
\tr(T_0^{(s)} T_f^t) \cong \widetilde{f}^{(s)}(0).
\end{align*}
More generally,
\begin{align*}
\tr(T_0^{(s)}W_{-z}^t T_f^t W_z^t) = \tr(T_0^{(s)} T_{\alpha_{-z}(f)}^t) \cong \widetilde{f}^{(s)}(z).
\end{align*}
Now, using the standard estimate $|\tr(AB)| \leq \| A\|_{\mathcal S^1} \| B\|_{op}$, where the norms denote the trace norm and operator norm, respectively, yields
\begin{align*}
|\widetilde{f}^{(s)}(z)| \lesssim \| T_0^{(s)}\|_{\mathcal S^1} \| T_f^t\|_{op}.
\end{align*}

In principle, all of this carries over upon replacing the trace ideal by the ideal of nuclear operators. First of all, we still have the identity $\tr(T_0^{(s)} T_f^t) \cong \widetilde{f}^{(s)}(0)$, which follows from the Dominated Convergence Theorem upon writing out the definition of the trace, provided that $T_0^{(s)}$ is indeed a nuclear operator (note that the nuclear trace is well-defined, as the Fock spaces possess the approximation property - see Section \ref{sec:aux} below for a short discussion of this). Those are indeed the two points which need clarification. For doing so, we first observe that the result on $F_t^\infty$ follows immediately from the result on $f_t^\infty$, as $\| T_f^t\|_{f_t^\infty \to f_t^\infty}$ and $\| T_f^t\|_{F_t^\infty \to F_t^\infty}$ are equivalent. Hence, we may exclude the case $\mathbb F_t = F_t^\infty$ in the following discussions for convenience.

It is therefore the first logical step to prove the following proposition.
\begin{prop}\label{prop:conv_t0}
Let $s > t/2$. Then, the series defining $T_0^{(s)}$ converges in nuclear norm. In particular, $T_0^{(s)} \in \mathcal N(\mathbb F_t)$.
\end{prop}
The proof hinges on the following fact:
\begin{lem}\label{lem:berger_coburn_proof}
Let $p \in [1, \infty]$. If $q\in [1, \infty]$ denotes the conjugate exponent, i.e.\ $1/p + 1/q = 1$, then
\begin{align*}
\sup_{\alpha \in \mathbb N_0^n} \| e_\alpha^t\|_{F_t^p} \| e_\alpha^t\|_{F_t^q} <\infty.
\end{align*}
\end{lem}
\begin{proof}
Using the tensor product structure of the standard basis, it suffices to consider the case $n = 1$. For $q = \infty$, basic calculus shows that
\begin{align*}
\| e_k^t\|_{F_t^\infty} &= \sup_{z \in \mathbb C} \frac{1}{\sqrt{k! t^k}} |z|^k e^{-\frac{|z|^2}{2t}}\\
&= \frac{1}{k!t^k} \sup_{r \geq 0} r^k e^{-\frac{r^2}{2t}}\\
&= \frac{1}{k!} k^{k/2} e^{-\frac{k}{2}}.
\end{align*}
For $p = 1$ we obtain
\begin{align*}
\| e_k^t\|_{F_t^1} &= \frac{1}{2\pi t} \int_{\mathbb C} \frac{|z|^k}{\sqrt{k! t^k}} e^{-\frac{|z|^2}{2t}}~dz\\
&= \frac{1}{t\sqrt{k!t^k}} \int_0^\infty r^{k+1} e^{-\frac{r^2}{2t}}~dr \\
&= \frac{2^{k/2}}{\sqrt{k!}} \Gamma \left( \frac{k}{2} +1 \right).
\end{align*}
Thus, we arrive at
\begin{align*}
\| e_k^t\|_{F_t^1} \| e_k^t\|_{F_t^\infty} = \frac{(2k)^{k/2}}{k!} \Gamma\left( \frac{k}{2} + 1 \right) e^{-\frac{k}{2}}.
\end{align*}
Stirling's approximation
\begin{align*}
\Gamma(x) = \sqrt{\frac{2\pi}{x}} \left( \frac{x}{e} \right)^x \left( 1 + \mathcal O \left( \frac{1}{x}\right) \right)\quad \text{as} \quad x \to \infty
\end{align*}
now easily gives
\begin{align*}
\sup_{k \in \mathbb N_0}\| e_k^t\|_{F_t^1} \| e_k^t\|_{F_t^\infty}< \infty.
\end{align*}
Clearly, the case $p = \infty$, $q = 1$ follows the same asymptotic behaviour. For $p \in (1, \infty)$, the uniform bound is immediately obtained from a version of Littlewood's interpolation inequality for Fock spaces.
\end{proof}
For completeness, we give the Fock space version of the interpolation inequality we just mentioned:
\begin{lem}
For $f \in F_t^1$ and $1 < p < \infty$ it holds true that:
\begin{align*}
\| f\|_{F_t^p} \leq p^{\frac{n}{p}} \| f\|_{F_t^1}^{1/p} \| f\|_{F_t^\infty}^{1 - 1/p}.
\end{align*}
\end{lem}
We omit the proof of this inequality, as it is a standard consequence of H\"{o}lder's inequality.
\begin{proof}[Proof of Proposition \ref{prop:conv_t0}]
Recall that the nuclear norm of $A \in \mathcal N(\mathbb F_t)$ is defined as
\begin{align*}
\| A\|_{\mathcal N} := \inf \left \{ \sum_j \| f_j\|_{\mathbb F_t} \| g_j\|_{(\mathbb F_t)'}: ~A = \sum_j f_j \otimes g_j, ~f_j \in \mathbb F_t, ~g_j \in (\mathbb F_t)' \right \}
\end{align*}
As usually, we will identify $(F_t^p)'$ with $F_t^q$, which are identical up to passing to an equivalent norm. Recall that we excluded the case $\mathbb F_t = F_t^\infty$ from our discussion.

We clearly have
\begin{align*}
\sum_{k=0}^\infty \left | 1 - \frac{t}{s}\right |^k \| P_k\|_{\mathcal N} &\leq \sum_{k=0}^\infty \left | 1 - \frac{t}{s}\right |^k \sum_{|\alpha| = k} \| e_{\alpha}^t\|_{\mathbb F_t} \| e_\alpha^t\|_{(\mathbb F_t)'}
\intertext{By Lemma \ref{lem:berger_coburn_proof}, we may uniformly estimate the product of the two norms to obtain}
&\lesssim \sum_{k=0}^\infty \left | 1 - \frac{t}{s}\right |^k  \cdot |\{ \alpha \in \mathbb N_0^n: ~|\alpha| = k\}|
\intertext{Some basic combinatorics shows}
&= \sum_{k=0}^\infty \left | 1 - \frac{t}{s}\right |^k \binom{k-1+n}{k}
\end{align*}
By the quotient test, this series converges. Therefore, the series defining $T_0^{(s)}$ converges absolutely in $\mathcal N(\mathbb F_t)$. As the nuclear ideal is complete, the series therefore converges and $T_0^{(s)} \in \mathcal N(\mathbb F_t)$.
\end{proof}
\begin{prop}
Let $f: \mathbb C^n \to \mathbb C$ such that $fK_z^t \in L_t^2$ for any $z \in \mathbb C^n$. Further, let $s \in (\frac{t}{2}, t)$. If $\alpha_{-z}(T_f^t)$ is bounded on any of the spaces $F_t^p$ or $f_t^\infty$, then
\begin{align*}
\widetilde{f}^{(s)}(z) = \left( \frac{t}{s} \right)^n \tr(T_0^{(s)}\alpha_{-z}(T_f^t)).
\end{align*}
\end{prop}
\begin{proof}
Without loss of generality, we may assume $z = 0$. Since the nuclear operators form an ideal, we have $T_0^{(s)}T_f^t \in \mathcal N(\mathbb F_t)$ provided that the Toeplitz operator is bounded. We may compute its nuclear trace as follows:
\begin{align*}
\tr(T_0^{(s)}T_f^t) &= \sum_{k=0}^\infty \left( 1 - \frac{t}{s} \right)^k \sum_{|\alpha| = k} \tr(P_k T_f^t)\\
&= \sum_{k=0}^\infty \left( 1 - \frac{t}{s} \right)^k \sum_{|\alpha| = k} \langle T_f^t e_\alpha^t, e_\alpha^t\rangle_{F_t^2}
\intertext{Note that for any polynomial $q$, $|q|^2$ can be dominated by a linear combination of functions $|K_z^t|^2$. Since $fK_z^t\in L_t^2$, we also get $fq \in L_t^2$, hence $T_f^t(q) = P_t(fq)$, thus:}
&= \sum_{k=0}^\infty \left( 1 - \frac{t}{s} \right)^k \sum_{|\alpha| = k} \langle f e_\alpha^t, e_\alpha^t\rangle_{F_t^2}\\
&= \frac{1}{(\pi t)^n} \sum_{\alpha \in \mathbb N_0^n} \sum_{\mathbb C^n} f(z) \frac{|z_1|^{2\alpha_1} \cdot \dots \cdot |z_n|^{2\alpha_n}}{\alpha!} \left( \frac{1}{t} - \frac{1}{s} \right)^{|\alpha|} e^{-\frac{|z|^2}{t}}~dz
\intertext{It is an easy exercise to show that $f \in L_t^2$ implies $f \in L_s^1$ whenever $s \in (t/2, t)$. Thus, we may apply the Dominated Convergence Theorem to obtain:}
&= \frac{1}{(\pi t)^n} \int_{\mathbb C^n} f(z)\sum_{\alpha \in \mathbb N_0^n}  \frac{|z_1|^{2\alpha_1} \cdot \dots \cdot |z_n|^{2\alpha_n}}{\alpha!} \left( \frac{1}{t} - \frac{1}{s} \right)^{|\alpha|} e^{-\frac{|z|^2}{t}}~dz\\
&= \left( \frac{s}{t} \right)^n \frac{1}{(\pi s)^n} \int_{\mathbb C^n} f(z) e^{-\frac{|z|^2}{s}}~ds\\
&= \left( \frac{s}{t}\right)^n \widetilde{f}^{(s)}(0).\qedhere
\end{align*}
\end{proof}
Combining all the steps, we have now seen that the Berger-Coburn estimates carry over to any Fock space, i.e.\ for $f$ with $fK_z^t \in L_t^2$ we have seen that
\begin{align*}
\| T_f^t\|  \lesssim & ~\| \widetilde{f}^{(s)}\|_\infty, \quad 0 < s < \frac{t}{2},\\
\| \widetilde{f}^{(s)}\|_\infty  \lesssim & ~\| T_f^t\|, \quad \frac{t}{2} < s < 2t,
\end{align*}
where the operator norm can be taken over any of the Fock spaces $F_t^p$, $1\leq p  \leq \infty$, and $f_t^\infty$. The Correspondence Theorem now immediately gives the following improvement of this result:
\begin{thm}\label{thm:berger_coburn_correspondence}
Let $f: \mathbb C^n \to \mathbb C$ such that $fK_z^t \in L_t^2$ for any $z \in \mathbb C^n$. Further, let $\mathcal D_0 \subset \operatorname{BUC}(\mathbb C^n)$ be a closed, $\alpha$-invariant subspace.
\begin{enumerate}[(1)]
\item If $\widetilde{f}^{(s)} \in \mathcal D_0$ for some $0 < s < \frac{t}{2}$, then $T_f^t \in \mathcal T_{lin}^{\mathbb F_t}(\mathcal D_0)$.
\item If $T_f^t \in \mathcal T_{lin}^{\mathbb F_t}(\mathcal D_0)$, then $\widetilde{f}^{(s)} \in \mathcal D_0$ for every $\frac{t}{2} < s < 2t$.
\end{enumerate}
\end{thm}
\begin{proof}
\begin{enumerate}[(1)]
\item By the Berger-Coburn estimate, $T_f^t$ is a bounded operator. As in the proof of Corollary \ref{cor:toep_c1}, $\widetilde{f}^{(s)} \in \operatorname{BUC}(\mathbb C^n)$ implies even $T_f^t \in \mathcal C_1(\mathbb F_t)$. Now, apply the Correspondence Theorem.
\item Recall that $g_s$ denotes the standard Gaussian $g_s(z) = \frac{1}{(\pi t)^n} e^{-\frac{|z|^2}{s}}$. By comparing Berezin transforms, it is not difficult to see that $g_s \ast T_f^t = T_{\widetilde{f}^{(s)}}^t$, where the star denotes the $L^1(\mathbb C^n)$ module structure of $\mathcal C_1(\mathbb F_t)$. Since $\mathcal T_{lin}^{\mathbb F_t}(\mathcal D_0)$ is, by definition, closed and $\alpha$-invariant, we get $T_{\widetilde{f}^{(s)}}^t \in \mathcal T_{lin}^{\mathbb F_t}(\mathcal D_0)$. Note that one can prove the following: A function $g \in \operatorname{BUC}(\mathbb C^n)$ is contained in $\mathcal D_0$ if and only if $\widetilde{g}^{(t)}$ is contained in $\mathcal D_0$ for some $t > 0$. This can be shown by mimicking the prove of the Correspondence Theorem, i.e.\ use that convolution by $g_s$ is an approximate identity on $\operatorname{BUC}(\mathbb C^n)$ and then, for suitable constants $c_j \in \mathbb C, ~z_j \in \mathbb C^n$, show that
\begin{align*}
\| g - \sum_j c_j \alpha_{z_j}(\widetilde{g}^{(t)})\|_\infty
\end{align*}
can be made arbitrarily small. Now, in our situation $T_{\widetilde{f}^{(s)}}^t \in \mathcal T_{lin}^{\mathbb F_t}(\mathcal D_0)$ implies that $\widetilde{f}^{(s+t)} = \widetilde{\widetilde{f}^{(s)}}^{(t)} \in \mathcal D_0$. Since $\widetilde{f}^{(s)}$ is in $\operatorname{BUC}(\mathbb C^n)$ by the Berger-Coburn estimates, the previous comment imply that $\widetilde{f}^{(s)} \in \mathcal D_0$.\qedhere
\end{enumerate}
\end{proof}
\begin{rem}
We want to emphasize that, after this theorem has been presented in the author's thesis, Wu and Zheng independently proved the first part of the theorem for $p = 2$ in their recent preprint \cite{Wu_Zheng2021}. Their methods differ significantly from our proof. For $\mathcal D_0 = C_0(\mathbb C^n)$ and $p = 2$, this result is already contained in \cite{Bauer_Coburn_Isralowitz2010}.
\end{rem}
The Berger-Coburn estimates naturally lead to the Berger-Coburn conjecture, which asks if $T_f^t$ is bounded if and only if $\widetilde{f}^{(t/2)}$ is bounded (with $f$ under the assumptions of the Berger-Coburn estimates). So far, this conjecture is only resolved for certain classes of symbols, most notably non-negative symbols and symbols of bounded mean oscillation (see e.g.\ \cite{Coburn_Isralowitz_Li2011} for some results concerning symbols of bounded mean oscillation). For such symbols, we have the following well-known fact:
\begin{lem}
Let $f: \mathbb C^n \to \mathbb C$ such that $fK_z^t \in L_t^2$ for any $z \in \mathbb C^n$. Further, assume that either $f \geq 0$ or $f \in \operatorname{BMO}(\mathbb C^n)$. Then, $\widetilde{f}^{(t)}$ is bounded for one $t > 0$ if and only if it is bounded for all $t > 0$.
\end{lem}
For this class of symbols, we now get the following correspondence statement:
\begin{thm}
Let $f: \mathbb C^n \to \mathbb C$ such that $fK_z^t \in L_t^2$ for any $z \in \mathbb C^n$ and let $\mathcal D_0$ be a closed, $\alpha$-invariant subspace of $\operatorname{BUC}(\mathbb C^n)$. Further, assume that either $f \geq 0$ or $f \in \operatorname{BMO}(\mathbb C^n)$. Then, the following are equivalent:
\begin{enumerate}
\item $T_f^t \in \mathcal T_{lin}^{\mathbb F_t}(\mathcal D_0)$;
\item $\widetilde{f}^{(t)} \in \mathcal D_0$;
\item $\widetilde{f}^{(s)} \in \mathcal D_0$ for any $s > 0$.
\end{enumerate}
\end{thm}
\begin{proof}
As we have already mentioned, the property of $\widetilde{f}^{(s)}$ being bounded is independent of $s> 0$ for such symbols. Therefore the property of $\widetilde{f}^{(s)}$ being in $\operatorname{BUC}(\mathbb C^n)$ is independent of $s > 0$ (simply write $\widetilde{f}^{(s)} = g_\varepsilon \ast \widetilde{f}^{(s-\varepsilon)}$). Now, $\widetilde{f}^{(t)} \in \mathcal D_0$ clearly implies $g_{s-t} \ast \widetilde{f}^{(t)} = \widetilde{f}^{(s)} \in \mathcal D_0$ for $s > t$. For $s < t$, we have seen in the proof of Theorem \ref{thm:berger_coburn_correspondence} that $\widetilde{f}^{(s)} \in \mathcal D_0$ if and only if $\widetilde{f}^{(s+t)} \in \mathcal D_0$. Summarizing, membership of $\widetilde{f}^{(s)}$ in $\mathcal D_0$ is independent of $s>0$. Thus, an application of Theorem \ref{thm:berger_coburn_correspondence} finishes the proof.
\end{proof}

\section{Correspondence of algebras}
In \cite{Fulsche2020}, we started the investigation of the following question: Let $\mathcal D_0$ be a closed, $\alpha$-invariant subspace of $\operatorname{BUC}(\mathbb C^n)$. Is there a relation between $\mathcal D_0$ being an algebra and $\mathcal T_{lin}^{\mathbb F_t}(\mathcal D_0)$ being an algebra?

As the author already noted in \cite[Corollary 3.9]{Fulsche2020}, the property of $\mathcal T_{lin}^{\mathbb F_t}(\mathcal D_0)$ being an algebra is independent of the particular choice of $\mathbb F_t$ (there, it was stated for $\mathbb F_t = F_t^p$ with $1 < p < \infty$, but this is true for every possible $\mathbb F_t$ by the same proof). Our investigations led to \cite[Theorem 3.13]{Fulsche2020}, which we shall not formulate here. Instead, we want to note that parts of the proof of said theorem needed some additional structure on $\mathcal D_0$, namely that it is self-adjoint and $\beta$-invariant ($\beta$ is the operator $\beta f(z) = f(-z)$). The assumption of $\beta$-invariance can be overcome within the frame of our initial proof, but our proof strictly hinges on the self-adjointness of $\mathcal D_0$. Nevertheless, in their recent preprint \cite{Wu_Zhao2021}, S.\ Wu and X.\ Zhao could give a better statement and proof in this particular direction, not using self-adjointness. Combining now all the available results, we arrive at the following theorem:
\begin{thm}[{\cite{Fulsche2020, Wu_Zhao2021}}]\label{corr:algebras}
Let $\mathcal D_0$ be a closed, $\alpha$-invariant subspace of $\operatorname{BUC}(\mathbb C^n)$. Then, the following are equivalent:
\begin{enumerate}[(1)]
\item $\mathcal D_0$ is an algebra;
\item $\mathcal T_{lin}^{\mathbb F_t}(\mathcal D_0)$ is an algebra for every $\mathbb F_t$ and every $t > 0$;
\item $\mathcal T_{lin}^{\mathbb F_t}(\mathcal D_0)$ is an algebra for some $\mathbb F_t$ and every $t > 0$.
\end{enumerate}
If $\mathcal D_0$ satisfies (1)-(3) above and $\mathcal I_0$ is a closed and $\alpha$-invariant subspace of $\mathcal D_0$, then the following are equivalent:
\begin{enumerate}[(1*)]
\item $\mathcal I_0$ is an ideal of $\mathcal D_0$;
\item $\mathcal T_{lin}^{\mathbb F_t}(\mathcal I_0)$ is a (left- or right-)ideal of $\mathcal T_{lin}^{\mathbb F_t}(\mathcal D_0)$ for every $\mathbb F_t$ and every $t > 0$;
\item $\mathcal T_{lin}^{\mathbb F:t}(\mathcal I_0)$ is a (left- or right-)ideal of $\mathcal T_{lin}^{\mathbb F_t}(\mathcal D_0)$ for some $\mathbb F_t$ and every $t > 0$.
\end{enumerate}
\end{thm}
We will add another contribution to this particular question. For this, we introduce the notation $\delta_\lambda f(z) = f(\lambda z)$ for any $\lambda > 0$. It is easy to see that $\delta_\lambda$ leaves $\operatorname{BUC}(\mathbb C^n)$ invariant for every $\lambda > 0$. For typographic reasons, we will reserve $\delta$ for the action on symbols and write $C_\lambda f(z) = f(\lambda z)$ for the operator on Fock space elements. Then, \cite[Lemma 2.6]{Zhu2012} shows that the linear operator $C_\lambda \in \mathcal L(\mathbb F_t, \mathbb F_{t/\lambda^2})$ is an isometric isomorphism.

The following statement is certainly well-known. Nevertheless, it seems that it is not contained in the standard reference \cite{Zhu2012}, hence we provide the simple proof.
\begin{lem}\label{lem:dil}
\begin{enumerate}[(1)]
\item Let $f \in L^\infty(\mathbb C^n)$. Then, holds true that
\begin{align*}
C_{1/\lambda} T_f^t C_\lambda = T_{\delta_{1/\lambda} f}^{t\lambda^2}.
\end{align*}
\item Let $A \in \mathcal L(\mathbb F_t)$ and $\lambda > 0$. Then,
\begin{align*}
\mathcal B(C_{1/\lambda} A C_\lambda) = \delta_{1/\lambda}(\mathcal B(A)).
\end{align*}
\end{enumerate}
\end{lem}
Here, we write the Berezin transform $\widetilde{A}$ for typographic reasons as $\mathcal B(A)$.
\begin{proof}
\begin{enumerate}[(1)]
\item The equality follows by direct computations:
\begin{align*}
C_{1/\lambda} T_f^t C_\lambda g(z) &= T_f^t C_\lambda g(z/\lambda)\\
&= \left( \frac{1}{\pi t} \right)^n \int_{\mathbb C^n} g(\lambda w) f(w) e^{\frac{z/\lambda \cdot \overline{w}}{t}} e^{-\frac{|w|^2}{t}}~dw
\intertext{We use the substitution $\lambda w = u, dw = \frac{1}{\lambda^2}du$:}
&= \left( \frac{1}{\pi t \lambda^2} \right)^n \int_{\mathbb C^n} g(u) f\left (\frac{u}{\lambda} \right ) \exp\left ( \frac{z \cdot \overline{u}}{t \lambda^2}-\frac{|u|^2}{t\lambda^2} \right )~du\\
&= T_{\delta_{1/\lambda} f}^{t\lambda^2} g(z)
\end{align*}
\item Recall that $C_{1/\lambda} A C_\lambda \in \mathcal L(\mathbb F_{t\lambda^2})$. Thus,
\begin{align*}
\mathcal B(C_{1/\lambda} A C_\lambda)(z) &= \langle C_{1/\lambda} A C_\lambda k_z^{t\lambda^2}, k_z^{t\lambda^2}\rangle_{F_{t\lambda^2}^2}\\
&= e^{-\frac{|z|^2}{t\lambda^2}} \langle A C_\lambda K_z^{t\lambda^2}, C_\lambda K_z^{t\lambda^2}\rangle_{F_t^2}.
\intertext{Using $C_\lambda K_z^{t\lambda^2} = K_{z/\lambda}^t$, which is readily verified, we obtain}
&= e^{-\frac{|z|^2}{t\lambda^2}} \langle A K_{z/\lambda}^t, K_{z/\lambda}^t\rangle_{F_t^2}\\
&= e^{-\frac{|z|^2}{t\lambda^2}} e^{\frac{|z|^2}{t\lambda^2}} \mathcal B(A)(z/\lambda)\\
&= \delta_{1/\lambda} \mathcal B(A)(z).\qedhere
\end{align*}
\end{enumerate}
\end{proof}
Adding $\delta$-invariance of $\mathcal D_0$, we can now fully characterize the property of $\mathcal T_{lin}^{\mathbb F_t}(\mathcal D_0)$ being an algebra. We say that a subset of $\operatorname{BUC}(\mathbb C^n)$ is $\delta$-invariant if it is invariant under $\delta_\lambda$ for every $\lambda > 0$.
\begin{thm}
Let $\mathcal D_0 \subset \operatorname{BUC}(\mathbb C^n)$ be $\alpha$-invariant and $\delta$-invariant. Then, the following are equivalent:
\begin{enumerate}[(1)]
\item $\mathcal D_0$ is a Banach algebra;
\item $\mathcal T_{lin}^{\mathbb F_t}(\mathcal D_0)$ is a Banach algebra for every $t > 0$;
\item $\mathcal T_{lin}^{\mathbb F_{t_0}}(\mathcal D_0)$ is a Banach algebra for one $t_0 > 0$.
\end{enumerate}
If $\mathcal D_0$ satisfies (1)-(3) and $\mathcal I_0 \subset \mathcal D_0$ is closed and both $\alpha$ and $\delta$-invariant, then the following are equivalent:
\begin{enumerate}[(1*)]
\item $\mathcal I_0$ is an ideal of $\mathcal D_0$;
\item $\mathcal T_{lin}^{\mathbb F_t}(\mathcal I_0)$ is a (left- or right-)ideal of $\mathcal T_{lin}^{\mathbb F_t}(\mathcal D_0)$ for every $t > 0$;
\item $\mathcal T_{lin}^{\mathbb F_{t_0}}(\mathcal I_0)$ is a (left- or right-)ideal of $\mathcal T_{lin}^{\mathbb F_{t_0}}(\mathcal D_0)$ for some $t_0 > 0$.
\end{enumerate}
\end{thm}
\begin{proof}
We show that (3) $\Rightarrow$ (2). Then, Theorem \ref{corr:algebras} will show the equivalence of (1), (2) and (3). Assume $\mathcal T_{lin}^{\mathbb F_{t_0}}(\mathcal D_0)$ is a Banach algebra for some fixed $t_0> 0$. Let $f_1, \dots, f_k \in \mathcal D_0$. Then, 
\begin{align*}
\mathcal B(T_{f_1}^{t_0\lambda^2} \dots T_{f_k}^{t_0\lambda^2}) &= \mathcal B(C_{1/\lambda} T_{f_1}^{t_0} C_\lambda C_{1/\lambda} \dots \dots C_{1/\lambda} T_{f_k}^{t_0} C_\lambda)\\
&= \mathcal B(C_{1/\lambda} T_{f_1}^{t_0} T_{f_2}^{t_0} \dots T_{f_k}^{t_0} C_\lambda)
\intertext{Using Lemma \ref{lem:dil} we get:}
&= \delta_{1/\lambda} \mathcal B(T_{f_1}^{t_0} \dots \mathcal T_{f_k}^{t_0}).
\end{align*}
Since we assumed that $\mathcal T_{lin}^{\mathbb F_{t_0}}(\mathcal D_0)$ is an algebra, we obtain
\begin{align*}
\mathcal B(T_{f_1}^{t_0} \dots \mathcal T_{f_k}^{t_0}) \in \mathcal D_0.
\end{align*}
Since $\mathcal D_0$ is $\delta$-invariant, we get
\begin{align*}
\delta_{1/\lambda} \mathcal B(T_{f_1}^{t_0} \dots \mathcal T_{f_k}^{t_0}) \in \mathcal D_0,
\end{align*}
which is now of course the same as
\begin{align*}
\mathcal B(T_{f_1}^{t_0\lambda^2} \dots T_{f_k}^{t_0\lambda^2}) \in \mathcal D_0.
\end{align*}
By the Correspondence Theorem, 
\begin{align*}
T_{f_1}^{t_0\lambda^2} \dots T_{f_k}^{t_0\lambda^2} \in \mathcal T_{lin}^{\mathbb F_{t_0\lambda^2}}(\mathcal D_0)
\end{align*}
for every $\lambda > 0$. Hence, $\mathcal T_{lin}^{\mathbb F_{s}}(\mathcal D_0)$ is a Banach algebra for every $s > 0$.

Using literally the same reasoning, one proves that (3*) implies (2*) and then concludes the proof.
\end{proof}

\section{On the Fredholm property}

As is well-known by now, the Fredholm property of operators from $\mathcal T^{\mathbb F_t}$ can be characterized in terms of \emph{limit operators}, at least for $F_t^p$ with $1 < p < \infty$ \cite{Fulsche_Hagger}. Unfortunately, the method of proof from \cite{Fulsche_Hagger} has an important drawback: It does not carry over to the non-reflexive cases. We will now discuss another approach to the problem, which goes through a better understanding of the limit structures, i.e.\ the \emph{algebra of limit operators}.

We want to start by briefly recalling the definition of limit functions and limit operators. We note that theory of limit operators has a long history on sequence spaces, see e.g.~\cite{Lindner2006}. The moral ancestors of the limit operators we use were first considered in \cite{Suarez2004} for the Bergman space over the complex ball and in \cite{Bauer_Isralowitz2012} over (reflexive) Fock spaces. Since the limit operators are really at the heart of the Fredholm theory, we describe them in some detail.

By $\mathcal M(\operatorname{BUC})$ we denote the maximal ideal space of the unital $C^\ast$ algebra $\operatorname{BUC}(\mathbb C^n)$. As is well-known, maximal ideal spaces of unital $C^\ast$ subalgebras of $C_b(\mathbb C^n)$ (the bounded continuous functions), which separate the points of $\mathbb C^n$ and contain $C_0(\mathbb C^n)$, are in 1:1 correspondence with compactifications of $\mathbb C^n$. This can be achieved by mapping each $z \in \mathbb C^n$ to the functional of point evaluation $\delta_z$. Then, $\{ \delta_z: ~z \in \mathbb C^n\}$ forms a dense subspace of $\mathcal M(\operatorname{BUC})$. 

In particular, any $f \in \operatorname{BUC}(\mathbb C^n)$ can be considered as a function $\Gamma(f) \in C(\mathcal M(\operatorname{BUC}))$ ($\Gamma$ being the Gelfand transform). In an abuse of notation, we will not distinguish between $f$ and $\Gamma(f)$ (identifying $f(z)$ with $\Gamma(f)(\delta_z)$) and will also write $f(x) := \Gamma(f)(x)$ for $x \in \mathcal M(\operatorname{BUC})$. For any $x \in \mathcal M(\operatorname{BUC})$ and $z \in \mathbb C^n$ we write
\begin{align*}
f_x(z) := x(\alpha_z(Uf)) = x(f(z - \cdot)).
\end{align*}
Let $(z_\gamma)_\gamma$ be a net in $\mathbb C^n$ converging to $x \in \mathcal M(\operatorname{BUC})$. Then, since $\mathcal M(\operatorname{BUC})$ is considered with the $w^\ast$ topology, we have
\begin{align*}
\alpha_{z_\gamma}(f)(w) = f(w -z_\gamma) = \delta_{z_\gamma}(f(w - \cdot)) \overset{\gamma}{\longrightarrow} f_x(w)
\end{align*}
pointwise and independently of the precise choice of the net $(z_\gamma)_\gamma$. An easy application of the Arzel\`{a}-Ascoli Theorem shows that this convergence is indeed uniform on compact subsets of $\mathbb C^n$ and further $f_x \in \operatorname{BUC}(\mathbb C^n)$ for any $x \in \mathcal M(\operatorname{BUC})$.

Given a net $(z_\gamma)_\gamma$ in $\mathbb C^n$ converging to $x \in \mathcal M(\operatorname{BUC})$, using the fact that $\alpha_z(T_f^t) = T_{\alpha_z(f)}^t$ it is easily seen that 
\begin{align*}
\alpha_{z_\gamma}(T_f^t) \overset{\gamma}{\longrightarrow} T_{f_x}^t
\end{align*}
in strong operator topology, at least for $\mathbb F_t \neq F_t^\infty$: If $q$ a holomorphic polynomial on $\mathbb C^n$, $\alpha_{z_\gamma}(T_f^t)q \to T_{f_x}^t q$ is readily established. Then, use that polynomials are dense and $\alpha_{z_\gamma} (T_f^t)$ is uniformly bounded in operator norm.

From this, it follows that for every $A \in \mathcal C_1(\mathbb F_t)$ and every $x \in \mathcal M(\operatorname{BUC})$, there exist $A_x$ such that $\alpha_{z_\gamma}(A) \to A_x$ in strong operator topology when $z_\gamma \to x$: Simply define $A_x$ as the limit of $T_{(f_k)_x}^t$ where the $f_k$ are such that $T_{f_k}^t \to A$ in operator norm. This all works flawless for $\mathbb F_t \neq F_t^\infty$. For $A \in \mathcal C_1(F_t^\infty)$, define $A_x$ as $((A|_{f_t^\infty})_x)^{\ast \ast}$.

Note that the group action of $\mathbb C^n$ induces an operation on $\mathcal M(\operatorname{BUC})$. Since the symbol $\alpha$ is reserved for the action on functions and operators, we denote this new action by $\tau$ and define it by letting $\tau_z(x)(f) := x(\alpha_z(f))$. Then, $\tau_z(x)$ is clearly again a multiplicative functional, i.e.\ $\tau_z(x) \in \mathcal M(\operatorname{BUC})$. Further, the group action $\beta f(w) = f(-w)$ induces the action $\nu$ on $\mathcal M(\operatorname{BUC})$ in the same way, i.e.\ by letting $\nu x(f) = x(\beta f)$. These group actions on $\mathcal M(\operatorname{BUC})$ are well-behaved in the sense that $\alpha_z(f_x) = f_{\tau_{-z}(x)}$ and $\alpha_z(A_x) = A_{\tau_{-z}(x)}$.

In \cite{Fulsche_Hagger}, we obtained the following characterization of the Fredholm property:
\begin{thm}[{\cite[Theorem 28]{Fulsche_Hagger}}]
Let $\mathbb F_t = F_t^p$ with $1 < p < \infty$ and $t > 0$. Then, $A \in \mathcal T^{\mathbb F_t}$ is Fredholm if and only if for every $x \in \mathcal M(\operatorname{BUC}) \setminus \mathbb C^n$, the limit operator $A_x$ is invertible.
\end{thm}

As already mentioned, we aim to generalize this result to any of the spaces $\mathbb F_t$ by different means. In the following, we will usually abbreviate $\mathcal M := \mathcal M(\operatorname{BUC})$ and $\partial \mathbb C^n := \mathcal M \setminus \mathbb C^n$.

\subsection{Some auxiliary facts}\label{sec:aux}
Recall that for $\mathbb F_t \neq F_t^\infty$, an operator $A \in \mathcal L(\mathbb F_t)$ is called \emph{sufficiently localized} (cf.\ \cite{Xia_Zheng,Isralowitz_Mitkovski_Wick}) if there are constants $C > 0$, $\beta > 2n$:
\begin{align}\label{eq:sufflocal}
|\langle Ak_z^t, k_w^t\rangle| \leq \frac{C}{(1 + |z-w|)^\beta}.
\end{align}
It is not hard to see, for $\mathbb F_t \neq F_t^\infty$ and $A \in \mathcal L(\mathbb F_t)$, $A$ is given by the integral expression
\begin{align}\label{eq:integralkernel}
Af(z) = \int_{\mathbb C^n} f(w) \langle A k_w^t, k_z^t\rangle ~d\mu_t(w).
\end{align}
Over $F_t^\infty$, we say that $A \in \mathcal L(F_t^\infty)$ is sufficiently localized if it satisfies both \eqref{eq:sufflocal} and \eqref{eq:integralkernel} (where the second condition is not automatically true, as the Berezin transform is not injective). As is well-known, Toeplitz operators with bounded symbols are sufficiently localized. Therefore, the Banach algebra $\mathcal A_{sl}^{\mathbb F_t}$ generated by sufficiently localized operators contains $\mathcal T^{\mathbb F_t}$. It was first proven in \cite{Xia} for $p = 2$ that $\mathcal A_{sl}^{F_t^2} = \mathcal T^{F_t^2}$. In \cite{Hagger2021}, the result was extended to $F_t^p$ for $1 < p < \infty$. Indeed, the proof given in \cite{Hagger2021} does not at all depend on the reflexivity and can verbatim be used in the case $p = 1$. Further, if $A$ is sufficiently localized over some $\mathbb F_t$, it is not hard to show that the integral operator \eqref{eq:integralkernel} defines is a bounded operator on each $\mathbb F_t$. This is a consequence of Young's inequality, cf.\ \cite{Hagger2021}. Since its formal adjoint, the integral operator
\begin{align*}
Bf(z) = \int_{\mathbb C^n} f(w) \overline{\langle Ak_z^t, k_w^t\rangle}~d\mu_t(w),
\end{align*}
is again sufficiently localized and therefore bounded on each $\mathbb F_t$, a standard duality argument generalizes the following result from \cite{Hagger2021} to each $\mathbb F_t$:
\begin{thm}\label{thm:suffloc}
For every $\mathbb F_t$ it holds true that $\mathcal T^{\mathbb F_t} = \mathcal A_{sl}^{\mathbb F_t}$.
\end{thm}

Let us recall a few facts on certain approximation properties. A Banach space $X$ is said to have the \emph{approximation property} (AP) if for each compact subset $K \subset X$ and each $\varepsilon > 0$ there exists an operator $T \in \mathcal L(X)$ of finite rank with
\begin{align*}
\| Tx - x\| \leq \varepsilon, \quad x \in K.
\end{align*}
This implies in particular that each compact operator can be approximated by finite rank operators in uniform topology.

$X$ has the $\lambda$-bounded approximation property ($\lambda$-BAP), where $\lambda \geq 1$, if there is a net $T_\gamma$ of finite rank operators on $X$, $\| T_\gamma\| \leq \lambda$ for every $\gamma$, with $T_\gamma \to \operatorname{Id}$ in SOT. Finally, $X$ has the $\lambda$-bounded compact approximation property ($\lambda$-BCAP) if there is a net $K_\gamma$ of compact operators on $X$, $\| K_\gamma \| \leq \lambda$, with $K_\lambda \to \operatorname{Id}$ in SOT.

An immediate consequence of these definitions is:
\begin{lem}
If $X$ has \emph{(AP)} and \emph{($\lambda$-BCAP)}, then it has \emph{($\lambda+\varepsilon$-BAP)} for every $\varepsilon > 0$.
\end{lem}

Since the classical $L^p$ spaces ($1 \leq p \leq \infty$) all have ($1$-BAP) (cf.\ \cite[Proposition 42]{Grothendieck1966}), it follows readily that $F_t^p$ ($1 \leq p \leq \infty$) has ($\| P_t\|$-BAP). $F_t^1 = (f_t^\infty)^\ast$ having (AP) implies that $f_t^\infty$ has (AP) (see e.g. \cite[Proposition 36]{Grothendieck1966}). In particular, all the Fock spaces have (AP). It will be important that $f_t^\infty$ also has the bounded approximation property - simultaneously, we will improve the respective constants:
\begin{prop}
Let $\mathbb F_t \neq F_t^\infty$. Then, $\mathbb F_t$ has \emph{($(1+\varepsilon)$-BAP)} for each $\varepsilon > 0$.
\end{prop}
\begin{proof}
We will show that $\mathbb F_t$ has ($(1+\varepsilon)$-BCAP) for every $\varepsilon > 0$. Together with (AP) and the previous lemma, this implies the desired result.

For $s < 1$ set $K_s f(z) = f(sz)$. Then, for $1 \leq p < \infty$:
\begin{align*}
\| K_s f\|_{F_t^p}^p &= \left(\frac{p}{2\pi t}\right)^n \int_{\mathbb C^n} |f(sz)|^p e^{-\frac{p|z|^2}{2t}}~dz\\
&= \frac{1}{s^{2n}} \left(\frac{p}{2\pi t}\right)^n \int_{\mathbb C^n} |f(w)|^p e^{-\frac{p|w|^2}{2ts^2}}~dw\\
&\leq \frac{1}{s^{2n}} \left(\frac{p}{2\pi t}\right)^n \int_{\mathbb C^n} |f(w)|^p e^{-\frac{p|w|^2}{2t}}~dw\\
&= \frac{1}{s^{2n}} \| f\|.
\end{align*}
Hence, $\| K_s\| \leq \left(\frac{1}{s^{2n}}\right)^{1/p}$ over $F_t^p$. Another simple estimate shows that $\| K_s\| \leq 1$ over $F_t^\infty$ (hence also over $f_t^\infty$).

We will show that $K_s$ is compact. For this, we will show that $K_s \in \mathcal T^{\mathbb F_t}$. It is then an easy exercise to compute $\widetilde{K_s}(z) = e^{-\frac{|z|^2}{t} + \frac{s|z|^2}{t}} \in C_0(\mathbb C^n)$, which shows the compactness.

Membership of $K_s$ in $\mathcal T^{\mathbb F_t}$ will be established by showing that $K_s$ is sufficiently localized. Then, the statement is a consequence of Theorem \ref{thm:suffloc}. We compute:
\begin{align*}
|\langle K_s k_z^t, k_w^t\rangle| &= e^{-\frac{|z|^2 + |w|^2}{2t}} |\langle K_{sz}^t, K_w^t\rangle| = e^{-\frac{|z|^2 + |w|^2}{2t}} e^{\frac{\re(s w\cdot \overline{z})}{t}}\\
&= e^{-\frac{1}{s}\frac{|\sqrt{s}z|^2 + |\sqrt{s}w|^2}{2t}} e^{\frac{\re((\sqrt{s}w) \cdot (\overline{\sqrt{s}z})}{t}}\\
&\leq e^{-\frac{|\sqrt{s}z|^2 + |\sqrt{s}w|^2}{2t}} e^{\frac{\re((\sqrt{s}w) \cdot (\overline{\sqrt{s}z})}{t}}\\
&= e^{-\frac{|\sqrt{s}z - \sqrt{s}w|^2}{2t}} = e^{-s\frac{|z-w|^2}{2t}}\\
&\lesssim \frac{1}{(1+|z-w|)^\beta}
\end{align*}
Finally, note that $K_s \to \operatorname{Id}$ in SOT as $s \uparrow 1$. This is well-known, as e.g.\ shown in \cite[Proposition 2.9]{Zhu2012} for $F_t^p$, $1 \leq p < \infty$. Over $f_t^\infty$, this is an easy exercise. Therefore, $(K_s)_{1 > s \geq s_0}$ is a net as needed in the definition of ($(1+ \varepsilon)$-BCAP) for $s_0 \in (0,1)$ appropriate.
\end{proof}

\subsection{The algebra of limit operators}\label{subsec:alglim}
\begin{defn}
A \emph{compatible family of limit operators} is a map $\gamma: \partial \mathbb C^n \to \mathcal T^{\mathbb F_t}$, satisfying the following properties:
\begin{enumerate}[(1)]
\item $x \mapsto \langle \gamma(x)1, 1\rangle$ is continuous.
\item For every $x \in \partial \mathbb C^n$ and $z \in \mathbb C^n$:
\begin{align*}
\alpha_z(\gamma(x)) = \gamma(\tau_{-z}(x)).
\end{align*}
\item We have $\sup_{x\in \partial \mathbb C^n} \| \gamma(x) \| < \infty$.
\item The family of operator-valued functions
\begin{align*}
\{ \mathbb C^n \ni z \mapsto \alpha_z(\gamma(x)): ~x \in \partial \mathbb C^n \}
\end{align*}
is uniformly equicontinuous with respect to the operator norm.
\end{enumerate}
\end{defn}
Here, we say that a family of operator-valued functions $(f_i)_{i\in I}$, where $f_i: \mathbb C^n \to \mathcal L(\mathbb F_t)$, is uniformly equicontinuous, if for each $\varepsilon > 0$ there exist $\delta > 0$ such that for $z, w \in \mathbb C^n$ with $|w| < \delta$ we have
\begin{align*}
\| f_i(z) - f_i(z-w)\|_{\mathbb F_t \to \mathbb F_t} < \varepsilon
\end{align*}
for every $i \in I$.

Of course, limit operators give a compatible family of limit operators:
\begin{lem}
Let $A \in \mathcal T^{\mathbb F_t}$. Then, $\gamma(x) = A_x$, $x \in \partial \mathbb C^n$, is a compatible family of limit operators.
\end{lem}
\begin{proof}
Continuity of the map $x \mapsto \langle A_x 1, 1\rangle$ follows from continuity in strong operator topology, which is obtained directly from the construction of the limit operators (for $\mathbb F_t =F_t^\infty$, note that $\langle A_x 1, 1\rangle = \langle (A|_{f_t^\infty})_x 1, 1\rangle$). Also, we have (using that the Weyl operators are isometries),
\begin{align*}
\| A_x\| &= \sup_{\| f\| = 1} \| A_x f\| = \sup_{\| f\| = 1} \| \lim_{\gamma} \alpha_{z_\gamma}(A)f\|\\
&= \sup_{\| f\| = 1} \limsup_{\gamma} \| AW_{-z_\gamma} f\| \leq \sup_{\| f\| = 1} \limsup_{\gamma} \| A\| \| W_{-z_\gamma} f\| = \| A\|.
\end{align*}
If $(w_\gamma)$ is a net in $\mathbb C^n$ converging to $x \in \partial \mathbb C^n$, we have
\begin{align*}
\alpha_z(A_x) &= \alpha_z(\lim_{\gamma} \alpha_{w_\gamma}(A)) = \lim_{\gamma} \alpha_{z}(\alpha_{w_\gamma}(A))\\
&= \lim_\gamma \alpha_{w_\gamma + z}(A) = A_{\alpha_{-z}(x)},
\end{align*}
where in the last step we possibly need to pass to a subnet. Finally, uniform equicontinuity follows from the estimate
\begin{align*}
\| \alpha_w(A_x) - A_x\| = \| (\alpha_z(A) - A)_x\| \leq \| \alpha_w(A) - A\|
\end{align*}
and $\| \alpha_z(A_x) - \alpha_{z-w}(A_x)\| = \| A_x - \alpha_{-w}(A_x)\|$.
\end{proof}
By $\mathfrak{lim} \mathcal T^{\mathbb F_t}$ we will denote the set of all compatible families of limit operators. Endowed with the pointwise sum, this obviously turns into a complex vector space. Indeed, it even turns into an algebra upon endowed with the pointwise product, but this is not clear a priori. What is clear is that the pointwise adjoint and pre-adjoint of a compatible family of limit operators is again a compatible family of limit operators. We endow $\mathfrak{lim}\mathcal T^{\mathbb F_t}$ with the norm
\begin{align*}
\| \gamma\|_{\mathfrak{lim}} = \sup_{x \in \partial \mathbb C^n} \| \gamma(x)\|.
\end{align*}
The following fact is now immediate:
\begin{lem}\label{lem:limit_alg}
The range of the map
\begin{align*}
\mathcal T^{\mathbb F_t} \ni A \mapsto [\gamma_A(x) = A_x] \in \mathfrak{lim} \mathcal T^{\mathbb F_t}
\end{align*}
is a normed algebra. The map is homomorphism of normed algebras onto its range, which also is a contraction with kernel $\mathcal T\mathcal K$.
\end{lem}
By passing to the quotient $\mathcal T^{\mathbb F_t}/\mathcal T \mathcal K$, we obtain an injective map
\begin{align*}
\mathcal T^{\mathbb F_t}/\mathcal T \mathcal K \ni A + \mathcal T\mathcal K \mapsto [\gamma_A(x) = A_x] \in \mathfrak{lim} \mathcal T^{\mathbb F_t},
\end{align*}
which is again a contractive normed algebra homomorphism.
\begin{prop}\label{prop:essnorm}
Let $\mathbb F_t \neq f_t^\infty, F_t^\infty$. Then, we have:
\begin{align*}
\sup_{x \in \partial \mathbb C^n} \| A_x\| \leq \| A + \mathcal T\mathcal K\|_{\mathcal T^{\mathbb F_t}/\mathcal T \mathcal K} \leq \sup_{x \in \partial \mathbb C^n} \| A_x\| \| P_t\|.
\end{align*}
\end{prop}
\begin{proof}
The first inequality is clear.

The second inequality actually follows from the methods presented in \cite{Fulsche_Hagger}, at least for $1 \leq p < \infty$. That paper was only written on the reflexive case and indeed some of its methods do not carry over to the general case (not even $p = 1$). Nevertheless, the proof of Theorem 31 (and all the statements on which its proof is based) extend naturally to $p= 1$, showing that
\begin{align*}
\| A + \mathcal T \mathcal K\| \leq \sup_{x \in \partial \mathbb C^n} \| A_x\| \| P_t\|
\end{align*}
for $\mathbb F_t=  F_t^p$ with $1 \leq p < \infty$.
\end{proof}
\begin{cor}\label{cor:essnorm}
For each $A \in \mathcal T^{\mathbb F_t}$ we have
\begin{align*}
\sup_{x \in \partial \mathbb C^n} \| A_x\| \simeq \| A + \mathcal T \mathcal K\|_{\mathcal T^{\mathbb F_t}/\mathcal T \mathcal K} .
\end{align*}
\end{cor}
\begin{proof}
In the reflexive cases, the statement is contained in the previous proposition. Over $f_t^\infty$, we of course have $\mathcal T \mathcal K = \mathcal K$, hence
\begin{align*}
\| A + \mathcal K\| = \| A + \mathcal T \mathcal K\|_{\mathcal T^{\mathbb F_t}/\mathcal T \mathcal K}.
\end{align*}
Since every operator from $\mathcal T^{F_t^1}$ has a pre-adjoint in $\mathcal T^{f_t^\infty}$, we have
\begin{align*}
\| A^\ast + \mathcal T\mathcal K\|_{\mathcal T^{F_t^1}/\mathcal T \mathcal K} \simeq \| A + \mathcal T \mathcal K\|_{\mathcal T^{f_t^\infty}/\mathcal T \mathcal K}
\end{align*}
and similarly
\begin{align*}
\| A^{\ast \ast} + \mathcal T \mathcal K\|_{\mathcal T^{F_t^\infty}/\mathcal T \mathcal K} \simeq \| A + \mathcal T \mathcal K\|_{\mathcal T^{f_t^\infty}/\mathcal T \mathcal K}
\end{align*}
for $A \in \mathcal T^{f_t^\infty}$. Further, since $F_t^1$ and $F_t^\infty$ have the ($\lambda$-BAP) for every $\lambda > 1$, \cite[Theorem 3]{Axler_Jewell_Shields1980} shows that
\begin{align*}
\| A + \mathcal K(f_t^\infty)\|_{\mathcal L(f_t^\infty)/\mathcal K(f_t^\infty)} &\simeq \| A^\ast + \mathcal K(F_t^1)\|_{\mathcal L(F_t^1)/\mathcal K(F_t^1)}\\
&\simeq \| A^{\ast \ast} + \mathcal K(F_t^\infty)\|_{\mathcal L(F_t^\infty)/\mathcal K(F_t^\infty)}
\end{align*}
Now, an application of the previous proposition for the case $p = 1$ finishes the proof.
\end{proof}
\subsection{Reconstructing operators from limit operators}
Given a function $f \in \operatorname{BUC}(\mathbb C^n)$, it is not hard to see that any other $g \in \operatorname{BUC}(\mathbb C^n)$ with $g|_{\partial \mathbb C^n} = f|_{\partial \mathbb C^n}$ is of the form $g = f + h$ for some $h \in C_0(\mathbb C^n)$. Let us reformulate this simple fact and fix it as a lemma:
\begin{lem}\label{lem:reconstr_funct}
Given $f \in C(\partial \mathbb C^n)$, there is some $f_0 \in \operatorname{BUC}(\mathbb C^n)$ with $f_0|_{\partial \mathbb C^n} = f$. Further, $f_0$ is unique modulo $C_0(\mathbb C^n)$.
\end{lem}
\begin{proof}
$\partial \mathbb C^n$ is a closed subspace of $\mathcal M$, which is compact and Hausdorff, hence normal. Thus, Tietze's extension theorem allows us to extend $f \in C(\partial \mathbb C^n)$ to $f_0 \in C(\mathcal M)$. But any function in $C(\mathcal M)$ is in $\operatorname{BUC}(\mathbb C^n)$. The uniqueness modulo $C_0(\mathbb C^n)$ has already been mentioned above.
\end{proof}
Note that, in the notation of the lemma, $f$ contains all the information about the limit functions of $f_0$: $(f_0)_x(w)$ is simply $\beta f( \tau_w(x))$. Hence, the above lemma says that we can reconstruct functions from $\operatorname{BUC}(\mathbb C^n)$ from their limit functions, at least modulo $C_0(\mathbb C^n)$. The same is indeed true for operators from $\mathcal T^{\mathbb F_t}$, nevertheless the proof needs more explanation.

The principal result in the discussion of compatible families of limit operators is the following:
\begin{thm}\label{thm:reconstr_limit}
Let $\gamma$ be a compatible family of limit operators. Then, there exists $A \in \mathcal T^{\mathbb F_t}$ such that $\gamma(x) = A_x$ for every $x \in \partial \mathbb C^n$. $A$ is unique modulo $\mathcal T \mathcal K$.
\end{thm}

The proof will make use of the following lemma:
\begin{lem}
Let $S \subset \mathcal T^{\mathbb F_t}$ be a norm bounded subset such that the family of functions
\begin{align*}
\{ z\mapsto \alpha_z(A): ~A \in S\}
\end{align*}
is uniformly equicontinuous with respect to the operator norm, i.e.\ for each $\varepsilon> 0$ there exist $\delta > 0$ such that
\begin{align*}
\sup_{A \in S} \| \alpha_z(A) - A\| < \varepsilon
\end{align*} 
for $|z| < \delta$. For each $N \in \mathbb N$, let $c_j^N \in \mathbb C$, $z_j^N \in \mathbb C^n$ be finitely many constants such that
\begin{align*}
\| g_{t/N} - \sum_j c_j^N \alpha_{z_j^N}(g_t)\|_{L^1} \leq \frac{1}{N}.
\end{align*}
Then, the convergence 
\begin{align*}
\| A - \sum_j c_j^N \alpha_{z_j^N}(T_{\widetilde{A}}^t)\| \to 0, \quad N \to \infty
\end{align*}
is uniform in $A \in S$.
\end{lem}
\begin{proof}
As in the proof of the Correspondence Theorem, we have
\begin{align*}
\| A - \sum_j c_j^N \alpha_{z_j^N}(T_{\widetilde{A}}^t)\| \leq \| A - g_{t/N} \ast A\| + \| g_{t/N} - \sum_j c_j \alpha_{z_j}(g_t)\|_{L^1} \| A\|.
\end{align*}
The second term clearly converges uniformly to $0$, as $\sup_{A \in S} \| A\| < \infty$. For the first term, observe that
\begin{align*}
\| A - g_{t/N} \ast A\| \leq \int_{\mathbb C^n} g_{t/N}(w) \| A -  \alpha_w(A)\| ~dw
\end{align*}
Now, for $\varepsilon > 0$ let us choose $\delta > 0$ such that
\begin{align*}
\| A - \alpha_w(A)\| < \varepsilon
\end{align*}
for every $A \in S$. Pick $N$ large enough such that
\begin{align*}
\int_{B(0, \delta)^c} g_{t/N}(w) ~dw < \varepsilon.
\end{align*}
Then, 
\begin{align*}
\int_{\mathbb C^n} g_{t/N}(w) \| A -  \alpha_w(A)\| ~dw &\leq 2\varepsilon\| A\| + \int_{B(0, \delta)} g_{t/N}(w) \| A - \alpha_w(A)\|~dw\\
&\leq 2\varepsilon\| A\| + \varepsilon \int_{B(0, \delta)} g_{t/N}(w)~dw\\
&\leq \varepsilon(1 + 2\| A\|).
\end{align*}
This proves uniform convergence of the approximation scheme.
\end{proof}
\begin{proof}[Proof of Theorem \ref{thm:reconstr_limit}]
Consider the function $f(x) = \langle \gamma(x)1, 1\rangle$. Then, $f \in C(\partial \mathbb C^n)$. Note that 
\begin{align*}
\widetilde{\gamma(x)}(z) = \langle \alpha_{-z}(\gamma(x)) 1, 1\rangle = \langle \gamma(\tau_{z}(x)) 1, 1\rangle = f(\tau_z(x)).
\end{align*}
Pick a function $f_0 \in \operatorname{BUC}(\mathbb C^n)$ according to Lemma \ref{lem:reconstr_funct} such that $f_0|_{\partial \mathbb C^n} = f$. Let the constants $c_j^N$ and $z_j^N$ be as in the previous lemma. Note that the limit operators of
\begin{align*}
\sum_j c_j^N \alpha_{z_j^N}(T_{f_0}^t)
\end{align*}
are given by
\begin{align*}
\left[ \sum_j c_j^N \alpha_{z_j^N}(T_{f_0}^t) \right]_x = \sum_{j} c_j^N \alpha_{z_j^N} T_{(f_0)_x}^t,
\end{align*}
where the limit functions are 
\begin{align*}
(f_0)_x(w) = f(\nu(\tau_w(x)) = f(\tau_{-w}(\nu (x))).
\end{align*}
Since $\{ z \mapsto \alpha_z(\gamma(x)):~x \in \partial \mathbb C^n\}$ is uniformly equicontinuous, it is no problem to verify that $\{ z \mapsto \alpha_{-z}(\gamma(x)): ~x \in \partial \mathbb C^n\}$ is also uniformly equicontinuous. Therefore, the previous lemma implies that
\begin{align*}
\sup_{x \in \partial \mathbb C^n} &\| \beta(\gamma(\nu (x))) - \sum_j c_j^N \alpha_{z_j^N}(T_{\beta \widetilde{\gamma(\nu(x))}}^t)\| \\
&= \sup_{x \in \partial \mathbb C^n} \| \beta(\gamma(\nu(x))) - \sum_j c_j^N \alpha_{z_j^N} (T_{f(\tau_{-(\cdot)}(\nu (x)))}^t)\|\\
&= \sup_{x \in \partial \mathbb C^n} \| \beta(\gamma(\nu(x))) - \sum_j c_j^N \alpha_{z_j^N} (T_{(f_0)_x}^t)\|\\
&\to 0, \quad N \to \infty.
\end{align*}
Hence, the family of sequences
\begin{align*}
\{ (\sum_j c_j^N \alpha_{z_j^N} (T_{(f_0)_x}^t))_{N \in \mathbb N}: ~x \in \partial \mathbb C^n\}
\end{align*}
is uniformly Cauchy. 
Now, this implies by Corollary \ref{cor:essnorm} that the sequence
\begin{align*}
(\sum_j c_j^N \alpha_{z_j^N} (T_{f_0}^t) + \mathcal T \mathcal K)_{N\in \mathbb N}
\end{align*}
is Cauchy in $\mathcal T^{\mathbb F_t}/\mathcal T \mathcal K$. Let $A_0 + \mathcal T \mathcal K$ be the limit with some representative $A_0 \in \mathcal T^{\mathbb F_t}$. Then,
\begin{align*}
(A_0)_x = \lim_N \sum_j c_j^N \alpha_{z_j^N}(T_{(f_0)_x}^t) = \beta(\gamma(\nu(x)),
\end{align*}
hence for any net $z_\gamma \to x$:
\begin{align*}
\alpha_{z_\gamma}(\beta(A_0)) = \beta(\alpha_{-z_\gamma}(A_0)) \overset{\gamma}{\longrightarrow} \beta(\beta( \gamma(x))) = \gamma(x),
\end{align*}
i.e.\ $\gamma(x) = A_x$ with $A = \beta(A_0)$.
\end{proof}
The following is now an immediate consequence of Theorem \ref{thm:reconstr_limit}.
\begin{thm}\label{cor:quotientalgebra}
$\mathfrak{lim} \mathcal T^{\mathbb F_t}$ is a Banach algebra. The map
\begin{align*}
\mathcal T^{\mathbb F_t}/\mathcal T \mathcal K \ni A + \mathcal T \mathcal  K \mapsto [\gamma_A(x) = A_x] \in \mathfrak{lim} \mathcal T^{\mathbb F_t}
\end{align*}
is an isomorphism of unital Banach algebras.
\end{thm}
\begin{proof}
In Lemma \ref{lem:limit_alg} we already established that the map is an injective homomorphism of normed algebras onto its range. By Theorem \ref{thm:reconstr_limit}, the map is onto $\mathfrak{lim} \mathcal{T}^{\mathbb F_t}$. By Proposition \ref{prop:essnorm}, it is continuous with continuous inverse. Since $\mathcal T^{\mathbb F_t}/\mathcal T \mathcal K$ is complete, we also obtain that $\mathfrak{lim} \mathcal T^{\mathbb F_t}$ is complete.
\end{proof}

\subsection{Fredholmness through invertibility of limit operators}
Before coming to the Fredholm property, observe the following:
\begin{prop}
Let $A \in \mathcal T^{\mathbb F_t}$ with $A_x$ invertible for each $x \in \partial \mathbb C^n$. Then, it holds true that $\sup_{x\in \partial \mathbb C^n} \| A_x^{-1}\| < \infty$.
\end{prop}
\begin{proof}
For $1 < p < \infty$, the statement is contained in \cite{Fulsche_Hagger}. Again note that, while not every proof of that paper carries over to $p = 1$, all the ingredients needed for this particular statement can be proven verbatim. Hence, the same proof yields the proposition over $F_t^1$. Now, using the fact that $A_x$ is invertible if and only if $A_x^\ast$ is, and further $\| A_x^{-1}\| \simeq \| (A_x^{-1})^\ast\| = \| (A_x^\ast)^{-1}\|$, the statement follows also over $f_t^\infty$ and $F_t^\infty$.
\end{proof}
\begin{thm}
Let $A \in \mathcal T^{\mathbb F_t}$. Then, $A$ is Fredholm iff $A_x$ is invertible for every $x \in \partial \mathbb C^n$.
\end{thm}
\begin{proof}
We already know that $A$ being Fredholm is equivalent to $A$ being invertible in $\mathcal T^{\mathbb F_t}$ modulo $\mathcal T \mathcal K$. Clearly, invertibility in $\mathcal T^{\mathbb F_t}$ modulo $\mathcal T \mathcal K$ of $A$ implies that $A_x$ is invertible for every $x \in \partial \mathbb C^n$. Only the other implication needs a proof. 

Let $A_x$ be invertible for every $x \in \partial \mathbb C^n$. Then, all the inverses lie in $\mathcal T^{\mathbb F_t}$. Since $\alpha_z(A_x^{-1})$ is an inverse to $\alpha_z(A_x)$, we obtain $\alpha_z(A_x^{-1}) = A_{\tau_{-z}(x)}^{-1}$. Further, $x \mapsto A_x^{-1}$ is continuous in strong operator topology: This is a consequence of the second resolvent identity
\begin{align*}
A_x^{-1} - A_y^{-1} = A_x^{-1}(A_y - A_x) A_y^{-1}
\end{align*}
and the uniform boundedness of the $A_y^{-1}$, which follows from the previous proposition. Finally, note that $\{ z \mapsto \alpha_z(A_x^{-1}): ~x \in \partial \mathbb C^n\}$ is uniformly equicontinuous. This is a consequence of the standard estimate
\begin{align*}
\| \alpha_z(A_x^{-1}) - A_x^{-1}\| \leq \frac{\| A_x^{-1}\| \| \alpha_z(A_x) - A_x\|}{1 - \| \alpha_z(A_x) - A_x\| \| A_x^{-1}\|},
\end{align*}
following from the standard Neumann series argument, using $(\alpha_z(A_x))^{-1} = \alpha_z(A_x^{-1})$, and is valid for each $z$ with $\| \alpha_z(A_x) - A_x\| < \frac{1}{\sup_{y \in \partial \mathbb C^n} \| A_y^{-1}\|}$.

These facts show that $\gamma(x) = A_x^{-1}$ is a compatible family of limit operators. Therefore, Theorem \ref{thm:reconstr_limit} shows that there exists $B \in \mathcal T^{\mathbb F_t}$ with $B_x = A_x^{-1}$. Hence, $B_x A_x = A_x B_x = I$ for every $x \in \partial \mathbb C^n$, the compactness characterization for operators in $\mathcal T^{\mathbb F_t}$ yields that $A$ is Fredholm.
\end{proof}
Here are two consequences, which were already presented in \cite{Fulsche_Hagger} in the reflexive cases. The statements can now be obtained for arbitrary $\mathbb F_t$ with identical proofs.
\begin{cor}
Let $A \in \mathcal T^{\mathbb F_t}$. Then,
\begin{align*}
\sigma_{ess}(A) = \bigcup_{x \in \partial \mathbb C^n} \sigma (A_x).
\end{align*}
\end{cor}
Here, $\sigma_{ess}$ and $\sigma$ denote the essential spectrum and the spectrum, respectively.
\begin{cor} The following statements are true over any $\mathbb F_t$:
\begin{enumerate}[(1)]
\item Let $f \in \operatorname{VO}_\partial(\mathbb C^n)$. Then, 
\begin{align*}
\sigma_{ess}(T_f^t) = f(\partial \mathbb C^n).
\end{align*}
\item Let $f \in \operatorname{VMO}_\partial (\mathbb C^n)$. Then,
\begin{align*}
\sigma_{ess}(T_f^t) = \widetilde{f}^{(t)}(\partial \mathbb C^n).
\end{align*}
\end{enumerate}
\end{cor}
We want to emphasize that results concerning $\sigma_{ess}(T_f)$ with $f \in \operatorname{VMO}_\partial(\mathbb C^n)$ are available for a significantly larger class of Fock spaces, e.g.\ non-reflexive Fock spaces with non-Gaussian weights. We refer the interested reader to the literature, see e.g.\ \cite{Hu_Virtanen2020, Hu_Virtanen2022}.

\section{Discussion}

The above results clearly lead to some imminent questions, some of which we will briefly discuss now. First of all, having established formula for the essential spectra of operators from $\mathcal T^{F_t^p}$, one can wonder to which extend such spectral data depend on the value $p \in [1, \infty]$. More precisely, we consider Toeplitz operators: Given $f \in L^\infty(\mathbb C^n)$, $T_f^t$ acts continuously on $F_t^p$ for each $p \in [1, \infty]$, so one can wonder if $\sigma(T_f^t)$ or $\sigma_{ess}(T_f^t)$ depend on the choice of the space $F_t^p$ on which the operator is realized. Clearly, the results for symbols of vanishing oscillation show that $\sigma_{ess}(T_f^t)$ is independent of $p$ in that particular case. Indeed, there is an algebra $\mathcal W_t$ of linear operators, containing Toeplitz operators with bounded symbols, which is densely contained in $\mathcal T^{F_t^p}$ for each $p \in [1, \infty]$, such that for each $A \in \mathcal W_t$ it holds true that $\sigma(A)$, $\sigma_{ess}(A)$ and $\operatorname{ind}(A)$ (provided the latter exists) are independent of the choice of the parameter $p$. Proving this relies on the results of the present work, as well as some other deep results. Details can be found in \cite{Fulsche2023}. In particular, these results show that the study of (essential) spectra as well as Fredholm indices of Toeplitz operators reduce to the Hilbert space case.

In recent years, Fock spaces $F^p(\varphi)$ with non-Gaussian weights $\varphi$ and operator theory on them have been in the focus of quite a number of publications. In general, one should not expect the methods of the present paper to carry over to such spaces: At the heart of our method lies the CCR relation $W_z^t W_z^t = e^{-i\sigma(z,w)/t}W_{z+w}$, satisfied by the Weyl operators. Of course, this relation is a consequence of the rather special form of the (normalized) reproducing kernel functions, which one loses when changing to other weights. Nevertheless, it seems not entirely unlikely that there exists a class of non-Gaussian weights such that at least parts of the methods presented here carry over to a more general setting. Investigating this should be an interesting problem for future work.

Besides Toeplitz operators, Hankel operators are certainly among the best understood operators on Fock spaces. Hankel operators have been investigated on a vast class of Fock spaces, see e.g.\ the recent results in \cite{Hu_Virtanen2022a, Hu_Virtanen2022a_corrigendum, Hu_Virtanen2023}. Nevertheless, investigations of Hankel operators on Fock spaces $F^p(\varphi)$ usually rule out the case $p = \infty$. It could interesting to check if any of the methods provided in this paper are suitable to obtain results on Hankel operators on $F_t^\infty$. As the paper \cite{Hagger_Virtanen2021} showed, this is not entirely hopeless, as limit operator methods are indeed suitable to study properties of Hankel operators.

Last but not least, we want to mention the analogous problems on the Bergman space (say, on the unit ball $\mathbb B_n$ of $\mathbb C^n$). Indeed, many of the results we have proven here have analogous results in that setting  - at least in the Hilbert space case and to some extend also on the reflexive Bergman spaces $A_\nu^p(\mathbb B_n)$ with standard weights (cf.\ \cite{Hagger2017, Mitkovski_Suarez_Wick2013, Xia}). While having similar results to ours, the present methods so far could not be adapted to the case of the Bergman space in a satisfactory way. Obtaining any progress in this direction would be desirable.

\section{Acknowledgement}
The author wishes to thank Raffael Hagger for a useful discussion on the Fredholm property.

\bibliographystyle{amsplain}
\bibliography{References}

\bigskip

\noindent
Robert Fulsche\\
\href{fulsche@math.uni-hannover.de}{\Letter fulsche@math.uni-hannover.de}
\\ 

\noindent
Institut f\"{u}r Analysis\\
Leibniz Universit\"at Hannover\\
Welfengarten 1\\
30167 Hannover\\
GERMANY

\end{document}